\documentclass[11pt,reqno]{amsart}
\usepackage[a4paper,margin=1in]{geometry}
\usepackage[T1]{fontenc}
\usepackage{lmodern,amsmath,amssymb,amsthm,mathtools,microtype}
\usepackage{tikz,needspace}
\usepackage[hidelinks]{hyperref}
\allowdisplaybreaks[1]
\numberwithin{equation}{section}
\newtheorem{theorem}{Theorem}[section]
\newtheorem{lemma}{Lemma}[section]

\newtheorem{proposition}{Proposition}[section]
\newtheorem{Coro}{Corollary}[section]
\theoremstyle{definition}

\theoremstyle{remark}
\newtheorem{remark}{Remark}[section]
\newtheorem*{remark*}{Remark}
\newcommand{\F}{\mathbb F}
\newcommand{\Aq}{\mathcal A_q}
\newcommand{\Atwo}{\mathcal A_{q^2}}
\newcommand{\Hq}{\mathcal H_{q^2}}
\newcommand{\Htwo}{\mathcal H_{q^2,g/3+1}}
\DeclareMathOperator*{\Res}{Res}

\newcommand{\dd}{\mathrm d}
\newcommand{\eps}{\varepsilon}

\newcommand{\norm}[1]{\lvert#1\rvert}
\title[The first moment of primitive quartic \(L\)-functions with a secondary term]{The first moment of primitive quartic \(L\)-functions with a secondary term}
\author{Ziwei Hong}
\address{Mathematical School, Renmin University of China, Beijing, P.R. China}
\address{School of Mathematics and Statistics, The University of New South Wales, Sydney, Australia}
\email{hongziwei@live.com}
\date{}
\subjclass[2020]{11M06, 11M41, 11N37, 11L05, 11L40}
\keywords{Central values, quartic $L$-functions, Gauss sums, function fields, double Dirichlet series}
\hypersetup{pdftitle={The first moment of central values of primitive quartic L-functions with fixed genus},pdfauthor={Ziwei Hong}}
\begin{document}

\begin{abstract}
Let $q\equiv3\pmod4$ be a fixed odd prime power. We study the first
moment of central values of primitive quartic Dirichlet $L$-functions
over $\F_q[T]$, with fixed genus $g$ and with $\chi^2$ primitive.
We obtain a main term $C_q q^{2g/3}$ and a secondary term
$(\alpha_r g+\beta_r)q^{g/2}$, with error
$O_{q,\eps}(q^{(16/33+\eps)g})$, where
$r\equiv g/3+1\pmod4$. The constants are explicit, with $C_q>0$
and $\alpha_r<0$ for every $r$. We use a double Dirichlet series
and the functional equation of the quartic Gauss sum generating
series to compute the residues and obtain the continuation needed
to separate both terms.
\end{abstract}

\maketitle

\section{Introduction}
\subsection{Background}
We study the first moment of central values of primitive quartic
Dirichlet $L$-functions over $\F_q[T]$, where $q\equiv3\pmod4$
is fixed. We average over characters of genus $g$ for which
$\chi^2$ is primitive with the same conductor as $\chi$.
Thus every ramified local character has order four.

Over number fields, Gao and Zhao~\cite{gao2021moments} obtained,
under the Lindel\"of hypothesis, an asymptotic formula for a smoothed first moment of primitive
quartic Dirichlet $L$-functions, using the method of Baier and
Young~\cite{BAIER2010879}. Under the generalized Lindel\"of
hypothesis, they subsequently treated first moments of fixed order
characters by double Dirichlet series in~\cite{GAO2024125}. Diaconu, Ion, Pa\c{s}ol and
Popa~\cite{DIPP2026} developed multiple Dirichlet series to study
first and second twisted moments of higher order Hecke
$L$-functions of order $r$ over fields containing the $2r$-th
roots of unity.
These results motivate the use of double Dirichlet series for
the non-Kummer function field family considered here.

In the function field setting, Andrade and
Keating~\cite{andrade2012mean} computed the first moment in the
hyperelliptic ensemble. Florea~\cite{florea2017improving}
improved the error term and separated a secondary term.
David, Florea and Lal\'in~\cite{2019The} obtained first moment
asymptotics for cubic characters in both the Kummer and
non-Kummer settings. Their analysis of the generating series of
cubic Gauss sums uses the functional equation arising from the
work of Hoffstein~\cite{1992Theta} and Patterson~\cite{Patterson}
to determine its residues and obtain upper bounds.
Their calculation also identifies a secondary contribution, but it
remains within the error term of the final asymptotic formula;
see~\cite[Remarks 4.5 and 5.6]{2019The}. A direct adaptation of
Florea's method does not give the error estimate needed to separate
this contribution for cubic characters, while applying the multiple
Dirichlet series method of~\cite{DIPP2026} would require substantially
more involved calculations. Hong and
Fang~\cite{hong2025meanvaluecubiclfuncitons} used a double
Dirichlet series in the non-Kummer cubic setting, with error
$O(q^{(7/8+\eps)g})$, which does not separate the secondary term.

\subsection{Main results}
For a character in our family, let $h$ denote its conductor and
$g$ the genus of the associated cyclic quartic cover. Following
Lin~\cite[Section~3]{lin2024one}, we have
\[
 \deg h=\frac{2g}{3}+2,\qquad
 m=\deg F=\frac g3+1,
\]
where $h=FF^{\sigma_0}$ in the parametrization of
Section~\ref{sec:prelim}.
Thus the admissible genera are $g\in3\mathbb Z_{\ge0}$, with
$m\ge1$. We let $g$ tend to infinity through this sequence and
write $r\equiv m\pmod4$, with $0\le r\le3$. Each character
and its complex conjugate are counted separately. Throughout the paper,
$g$ denotes the genus of the quartic cover; the reduced polynomial
$L_q(v,\chi)/(1-v)$ has degree $2m-2=2g/3$.
In the Euler products below, $P$ runs over monic irreducible polynomials
in $\F_q[T]$, and $p=|P|=q^{\deg P}$.

\begin{theorem}\label{main}
Let $q\equiv3\pmod4$ be a fixed odd prime power. As
$g\to\infty$ through $3\mathbb Z_{\ge0}$, for every $\eps>0$,
\begin{align}\label{eq:main-result}
 &\sum_{\substack{\chi\ \mathrm{primitive\ quartic}\\
 \chi^2\ \mathrm{primitive}\\ \operatorname{genus}(\chi)=g}}
 L_q\!\left(\frac12,\chi\right)\notag\\
 &\qquad=C_q q^{2g/3}
 +(\alpha_r g+\beta_r)q^{g/2}
 +O_{q,\eps}\!\left(q^{(16/33+\eps)g}\right),
\end{align}
where the positive main-term constant is
\begin{align}\label{eq:main-euler-product}
 C_q&=\frac{q^2}{1-q^{-1}}
 \prod_{\deg P\ \mathrm{odd}}(1-p^{-2})\notag\\
 &\quad\times\prod_{\deg P\ \mathrm{even}}
 (1+2p^{-1})(1-p^{-1})^2
 \left(1-\frac{2}{p^2(p+2)}\right),
\end{align}
and
\begin{equation}\label{eq:ab-genus}
 \alpha_r=\frac{q^{3/2}}{3}
 \sum_{\zeta^4=1}A_\zeta\zeta^{-r},\qquad
 \beta_r=q^{3/2}
 \sum_{\zeta^4=1}(2A_\zeta+B_\zeta)\zeta^{-r}.
\end{equation}
The constants $A_\zeta$ and $B_\zeta$ are given by convergent
Euler products, normally convergent correction series, and derivative
formulas in
Section~\ref{sec:residues}.
\end{theorem}

The products in \eqref{eq:main-euler-product} converge absolutely.
Section~\ref{PZ} derives this expression and identifies it with
$(q^2-1)P(q^{-2})Z(q^{-2},q^{-1/2})$ in the notation used there.
The secondary term comes from four double poles. Its linear
and constant coefficients depend only on $m$ modulo four.
Remark~\ref{rem:nonzero} shows that $\alpha_r<0$ for every $r$.
Since $16/33<1/2$, choosing $0<\eps<1/66$ separates both
coefficients of the secondary term from the error.

For the same quartic family, Hong~\cite{hong2026nonvanishing}
obtained a positive proportion of non-vanishing central values
using a uniform twisted first moment. The contribution here is
the separation of the secondary term and the improved error in
\eqref{eq:main-result}. The following weaker non-vanishing bound
is included as a direct, self-contained consequence of this
formula and the standard bounds for individual $L$-functions.
\begin{Coro}\label{non}
Under the assumptions of Theorem~\ref{main}, for every $\eps>0$,
\[
 \#\left\{\chi:\ \begin{array}{l}
 \chi\text{ and }\chi^2\text{ primitive},\quad
 \chi\text{ quartic},\\
 \operatorname{genus}(\chi)=g,\quad L_q(1/2,\chi)\ne0
 \end{array}\right\}
 \gg_{q,\eps}q^{(2/3-\eps)g}.
\]
\end{Coro}

\subsection{Main contributions and method of proof}
This paper has two main contributions. First, we compute the residues
of the quartic Gauss sum generating series and establish uniform upper
bounds for it. In Section~\ref{sec:gauss}, we follow the approach of
David, Florea and Lal\'in~\cite[Section~3]{2019The}, using the functional
equation arising from the work of Hoffstein and Patterson. These
results provide the analytic input for the double Dirichlet series
$A_4(s,w)$ introduced in Section~\ref{sec:generating}, from which we
obtain the first main term.

Second, we introduce an averaging argument that separates the secondary
term, including both coefficients of $(\alpha_r g+\beta_r)q^{g/2}$.
The new treatment, which distinguishes this part of our proof from
the first-moment arguments discussed above, is developed in
Sections~\ref{sec:D} and~\ref{sec:continuation}. We first keep the
M\"obius inversion and the finite sums that exclude prime divisors
as exact identities. The cancellation used to enlarge the convergence
region comes from summing jointly over the selected prime divisors
and the modulus in the Gauss sum, with their signs and character
factors retained. The remaining sum over $D$ is bounded absolutely.
In Section~\ref{sec:D}, orthogonality
of additive characters expresses the averages over polynomial
coefficients as character sums subject to trace congruences, which we
evaluate at prime powers. In Section~\ref{sec:continuation}, we obtain
two upper bounds by summing in different orders and retain the signs
when summing over the selected prime divisors. These estimates,
together with coefficient interpolation and Bochner's tube theorem,
give the continuation needed to move the coefficient contour beyond
the secondary poles. This separates the quartic analogue of the
secondary contribution observed by David, Florea and Lal\'in in the
cubic case, where it remained within their error term.

\subsection{Organization of the paper}
Section~\ref{sec:prelim} describes the non-Kummer quartic family and
records the functional equations, the properties of Gauss sums, and
Bochner's tube theorem. Section~\ref{sec:gauss} uses the functional
equation of the quartic Gauss sum generating series to compute its
residues and obtain upper bounds. Section~\ref{sec:generating} introduces
$A_4(s,w)$, establishes its first convergence regions, and computes the
residue giving the main term.

Section~\ref{sec:D} applies M\"obius inversion, records the formulas for
excluding prime divisors from the moduli in the Gauss sums, and evaluates the
averages over polynomial coefficients by additive character orthogonality
and local character sums. Section~\ref{sec:continuation} bounds these
averages, treats the principal characters and the finite local
restrictions, and proves analytic continuation by coefficient
interpolation and Bochner's tube theorem.
Section~\ref{sec:residues} computes the secondary residues, including
both coefficients of the polynomial multiplying $q^{g/2}$.
Finally, Section~\ref{sec:proof} moves the coefficient contour and proves
Theorem~\ref{main} and Corollary~\ref{non}.

\section{Preliminaries}\label{sec:prelim}
\subsection{Primitive quartic characters and \texorpdfstring{$L$}{L}-functions}
Throughout, $q\equiv3\pmod4$ is an odd prime power and $Q=q^2$.
Write $\sigma_0$ for the nontrivial automorphism of
$\mathbb F_Q/\mathbb F_q$, acting on polynomial coefficients.
Let $\mathcal A_q$ and $\mathcal A_{q^2}$ denote the sets of monic
polynomials over the respective fields, and let
$\mathcal H_{q^2}$ denote the squarefree polynomials in
$\mathcal A_{q^2}$. A subscript $n$ specifies degree $n$. All polynomial sums and products are over monic polynomials.
We write $|f|=q^{\deg f}$ and $|F|_2=|F|_Q=Q^{\deg F}$.

Fix an isomorphism $\Omega$ from the complex fourth roots of unity
to the fourth roots of unity in $\mathbb F_Q^\times$.
For a monic irreducible $\pi\in\mathbb F_Q[T]$, define
$\chi_\pi(a)=0$ when $\pi\mid a$; otherwise $\chi_\pi(a)$ is
determined by
\[
 \Omega(\chi_\pi(a))\equiv a^{(Q^{\deg\pi}-1)/4}\pmod\pi.
\]
For $F=\prod_\pi\pi^{e_\pi}$, put
$\chi_F=\prod_\pi\chi_\pi^{e_\pi}$. We use the same notation
for its restriction to $\mathbb F_q[T]$.
The primitive quartic characters whose squares are primitive
are precisely these restrictions with $F\in\mathcal H_{q^2}$,
$\deg F>0$ and $(F,F^{\sigma_0})=1$; their conductor is
$FF^{\sigma_0}$~\cite{2017On}.

For a Dirichlet character $\chi$ over $\mathbb F_q[T]$, define
\[
 L_q(s,\chi)=\sum_{f\in\mathcal A_q}\chi(f)|f|^{-s}.
\]
We also write $L_q(v,\chi)$ for the polynomial in $v=q^{-s}$
when $\chi$ is nonprincipal. Over $\mathbb F_Q[T]$ the analogous
function is denoted by $L_{q^2}$.

\begin{lemma}[{\cite[Lemma 3.6]{lin2024one}}]
For an admissible genus $g$,
\begin{equation}\label{equation}
 \sum_{\substack{\chi\ \mathrm{primitive\ quartic}\\
 \chi^2\ \mathrm{primitive}\\ \operatorname{genus}(\chi)=g}}
 L_q(1/2,\chi)
 =\sum_{\substack{F\in\mathcal H_{q^2,g/3+1}\\
 (F,F^{\sigma_0})=1}}L_q(1/2,\chi_F).
\end{equation}
\end{lemma}

Every character in this family is even: for
$a\in\mathbb F_q^\times$, the exponent $(Q^{\deg\pi}-1)/4$
is divisible by $q-1$, so $\chi_\pi(a)=1$.
For $m=\deg F$, the conductor has degree $2m$, and
$L_q(v,\chi_F)$ has degree $2m-1$ and contains the factor $1-v$.
Thus $L_q(v,\chi_F)/(1-v)$ has degree $2m-2$, whereas the
associated quartic cover has genus $3(m-1)$.
All moments in this paper use $L_q(1/2,\chi_F)$ itself.
For the reduced polynomial, the entire asymptotic formula is divided
by $1-q^{-1/2}$.

\begin{lemma}\label{fe}
Let $F\in\mathcal H_{q^2}$ have positive degree and satisfy
$(F,F^{\sigma_0})=1$. Then
\[
 L_q(s,\chi_F)=\omega(\chi_F)q^{2s-1}
 \frac{1-q^{-s}}{1-q^{s-1}}
 \frac{L_q(1-s,\overline{\chi_F})}{|F|_2^{s-1/2}},
\]
where
\[
 \omega(\chi_F)=-q^{1-\deg F}
 \sum_{\substack{f\in\mathcal A_q\\\deg f=2\deg F-1}}\chi_F(f).
\]
\end{lemma}
\begin{proof}
For a primitive even character of conductor $h$, the functional
equation is
\[
 L_q(s,\chi)=\omega(\chi)(q^{1/2-s})^{\deg h-2}
 \frac{1-q^{-s}}{1-q^{s-1}}L_q(1-s,\overline\chi).
\]
Here $\deg h=2\deg F$, giving the assertion.
\end{proof}

\begin{lemma}[{\cite[Lemmas 2.6 and 2.7]{2019The}}]\label{LLH}
Let $\chi$ be a nonprincipal primitive Dirichlet character of
conductor $h$ over $\mathbb F_q[T]$. For every $\varepsilon>0$,
\[
 |L_q(s,\chi)|\ll_{q,\varepsilon}q^{\varepsilon\deg h}
 \qquad (\Re s\ge1/2),
\]
and
\[
 |L_q(s,\chi)|\gg_{q,\varepsilon}q^{-\varepsilon\deg h}
 \qquad (\Re s\ge1).
\]
The same statements hold over $\mathbb F_Q[T]$, with $q$ replaced
by $Q$.
\end{lemma}

\subsection{Gauss sums}
Let $\ell$ be the characteristic of $\mathbb F_q$, and set
\[
 \psi_q(a)=\exp\left(\frac{2\pi i}{\ell}
 \operatorname{Tr}_{\mathbb F_q/\mathbb F_\ell}(a)\right).
\]
For a Laurent series $f=\sum_{j\le J}a_jT^j$ at infinity, define
\[
 e_q(f)=\psi_q(a_{-1}),\qquad
 e_{q^2}(f)=\psi_q\left(\operatorname{Tr}_{\mathbb F_Q/\mathbb F_q}
 (a_{-1})\right),
\]
over $\mathbb F_q$ and $\mathbb F_Q$, respectively. Thus $a_{-1}$
is the coefficient of $T^{-1}$.
For a character $\chi$ of conductor $h$, put
\[
 G_q(\chi)=\sum_{a\bmod h}\chi(a)e_q(a/h).
\]
For a residue symbol $\chi_f$, define
\[
 G_q(V,f)=\sum_{a\bmod f}\chi_f(a)e_q(aV/f),
 \qquad
 G_{q^2}(V,F)=\sum_{a\bmod F}\chi_F(a)e_{q^2}(aV/F).
\]
Here $f$ and $F$ are the moduli of the respective Gauss sums:
the variable $a$ runs through the residue classes modulo $f$ or $F$.
In particular, the modulus in $G_{q^2}(V,F)$ is its second argument
$F$. It need not be the conductor of $\chi_F$, since $\chi_F$ may be
imprimitive.
If $(a,f)=1$, changing variables gives
$G_q(aV,f)=\overline\chi_f(a)G_q(V,f)$, and similarly over
$\mathbb F_Q$.

For an odd character over $\mathbb F_q[T]$, put
\[
 \tau(\chi)=\sum_{a\in\mathbb F_q^\times}\chi(a)\psi_q(a),
 \qquad \epsilon(\chi)=q^{-1/2}\tau(\chi).
\]
Over $\mathbb F_Q[T]$, replace $q$ by $Q$ and $\psi_q$ by
$\psi_q\circ\operatorname{Tr}_{\mathbb F_Q/\mathbb F_q}$.
For even characters set $\epsilon(\chi)=1$.
\begin{Coro}[{\cite[Corollary 2.4]{2019The}}]\label{relation}
For a primitive character $\chi$ of conductor $h$ over
$\mathbb F_q[T]$,
\[
 \omega(\chi)=
 \begin{cases}
 \tau(\chi)^{-1}q^{-(\deg h-1)/2}G_q(\chi),&\chi\text{ odd},\\
 q^{-\deg h/2}G_q(\chi),&\chi\text{ even}.
 \end{cases}
\]
The analogous formula holds over $\mathbb F_Q[T]$.
\end{Coro}

\begin{lemma}\label{quartic-restriction}
Put $\kappa_q=(-1)^{(q+1)/4}$. If $F\in\mathcal H_{q^2}$ has
positive degree and $(F,F^{\sigma_0})=1$, then
\begin{equation}\label{quartic-gauss-transfer}
 G_q(\chi_F)=\chi_F(F^{\sigma_0})G_{q^2}(1,F)
 =\kappa_q^{\deg F}G_{q^2}(1,F).
\end{equation}
For squarefree $D\in\mathcal A_q$ and $(N,D)=1$, with
$N\in\mathcal A_q$, we also have
\begin{equation}\label{gauss-D-bound}
 G_{q^2}(N,D)=\kappa_q^{\deg D}q^{\deg D}.
\end{equation}
\end{lemma}
\begin{proof}
Write $h=FF^{\sigma_0}$. The Chinese Remainder Theorem parametrizes
the residue classes modulo $h$ over $\mathbb F_q$ by
\[
 a=\beta F^{\sigma_0}+\beta^{\sigma_0}F\pmod h,
 \qquad \beta\pmod F.
\]
For these representatives,
\[
 \chi_F(a)=\chi_F(\beta)\chi_F(F^{\sigma_0}),\qquad
 e_q(a/h)=e_{q^2}(\beta/F).
\]
Summing proves the first equality in \eqref{quartic-gauss-transfer}.
For $n=\deg F$ and $R=\operatorname{Res}(F,F^{\sigma_0})\ne0$,
the norm formula for the resultant gives
\[
 \Omega(\chi_F(F^{\sigma_0}))=R^{(q^2-1)/4}.
\]
Since $R^q=\operatorname{Res}(F^{\sigma_0},F)=(-1)^nR$,
\[
 R^{(q^2-1)/4}=(R^{q-1})^{(q+1)/4}
 =(-1)^{n(q+1)/4}=\kappa_q^n.
\]
The map $\Omega$ fixes $\{1,-1\}$, proving the second equality.

For \eqref{gauss-D-bound}, first take an irreducible
$P\in\mathbb F_q[T]$ of degree $d$ and put $p=q^d$.
If $d$ is odd, $P$ remains irreducible over $\mathbb F_Q$.
Its residue character on $\mathbb F_{p^2}$ is trivial on
$\mathbb F_p^\times$. For a nonzero element $\xi$ of trace zero,
the character sums over trace zero and trace $a\ne0$ are
$(p-1)\chi_P(\xi)$ and $-\chi_P(\xi)$, respectively.
Consequently the finite field Gauss sum is
$p\chi_P(\xi)=p(-1)^{(p+1)/4}=p\kappa_q^d$.
The derivative factor in the polynomial Gauss sum belongs to
$\mathbb F_p^\times$ and has character value one.
If $d$ is even, write $P=\pi\pi^{\sigma_0}$.
The two Gauss sums are complex conjugates and their product is
$p$. Their multiplicative factor is
$\chi_\pi(\pi^{\sigma_0})^2=1$, by the resultant calculation
above. Hence $G_{q^2}(1,P)=p=\kappa_q^dp$ in this case as well.
For rational polynomials the quartic symbol of a coprime rational
argument is one. Thus the factors between distinct prime divisors
of $D$ are one, and $\chi_D(N)=1$. Multiplicativity and a change
of variable now prove \eqref{gauss-D-bound}, including $D=1$.
\end{proof}

In the following lemma write $G_Q=G_{q^2}$ and
$|\pi|_2=Q^{\deg\pi}$. We have $Q\equiv1\pmod8$.
\begin{lemma}\label{Gauss1}
For polynomials over $\mathbb F_Q[T]$ the following identities hold.
\begin{enumerate}
\item If $(f_1,f_2)=1$, then
\[
 G_Q(V,f_1f_2)
 =\chi_{f_1}(f_2)^2G_Q(V,f_1)G_Q(V,f_2)
 =G_Q(Vf_2^2,f_1)G_Q(V,f_2).
\]
\item Let $V=V_1\pi^\alpha$, where $\pi$ is irreducible and
$\pi\nmid V_1$. For $i\ge1$,
\[
 G_Q(V,\pi^i)=
 \begin{cases}
 0,&i\le\alpha,\ 4\nmid i,\\
 \phi(\pi^i),&i\le\alpha,\ 4\mid i,\\
 -|\pi|_2^{i-1},&i=\alpha+1,\ 4\mid i,\\
 \epsilon(\chi_\pi^i)\omega(\chi_\pi^i)
 \chi_\pi^i(V_1^{-1})|\pi|_2^{i-1/2},
 &i=\alpha+1,\ 4\nmid i,\\
 0,&i\ge\alpha+2.
 \end{cases}
\]
Here $\phi$ is the polynomial Euler function. When $4\nmid i$,
the character $\chi_\pi^i$ is primitive modulo $\pi$.
\end{enumerate}
\end{lemma}
\begin{proof}
The Chinese Remainder Theorem writes each residue class modulo
$f_1f_2$ as $a_1f_1+a_2f_2$, with $a_1\bmod f_2$ and
$a_2\bmod f_1$. Substitution gives
\[
 G_Q(V,f_1f_2)=\chi_{f_2}(f_1)\chi_{f_1}(f_2)
 G_Q(V,f_1)G_Q(V,f_2).
\]
Quartic reciprocity and the change of variable formula prove
the first assertion.

If $i\le\alpha$, the additive character is trivial, so
\[
 G_Q(V,\pi^i)=|\pi|_2^{i-1}
 \sum_{a\bmod\pi}\chi_\pi(a)^i
 =\begin{cases}\phi(\pi^i),&4\mid i,\\0,&4\nmid i.
 \end{cases}
\]
If $i=\alpha+1$, writing $a=\pi b+c$ gives
\[
 G_Q(V,\pi^i)=|\pi|_2^{i-1}\chi_\pi^i(V_1^{-1})
 \sum_{c\bmod\pi}\chi_\pi(c)^i e_{q^2}(c/\pi).
\]
For $4\mid i$ the final sum is $-1$; otherwise
Corollary~\ref{relation}, with $q$ replaced by $Q$, evaluates it.
Finally, if $i\ge\alpha+2$, summing over the last free residue
coefficient gives zero by additive orthogonality.
\end{proof}

\subsection{Tube domains}
For an open set $U\subset\mathbb R^n$, write
$T(U)=U+i\mathbb R^n$. A domain of this form is called a tube
domain.
\begin{theorem}[{\cite[Theorem, p.~14]{Boc}}]\label{Bochner}
Let $U\subset\mathbb R^n$ be connected and open. Every
holomorphic function on $T(U)$ extends holomorphically to
$T(\operatorname{conv}U)$.
\end{theorem}
A numbered statement of the same tube theorem, with a proof, is
also given in~\cite[Theorem~1.1]{Noguchi2021}.

Write $\widehat T$ for the convex hull of a tube domain $T$.
\begin{theorem}[{\cite[Proposition C.5]{Cech1}}]
\label{Extending inequalities}
Let $g,h$ be holomorphic on a tube domain $T$, with $h$ nonzero
there, and let $\widetilde g,\widetilde h$ be their holomorphic
continuations to $\widehat T$. If $|g(z)|\le|h(z)|$ on $T$,
then $|\widetilde g(z)|\le|\widetilde h(z)|$ on $\widehat T$.
\end{theorem}

\section{The generating series of quartic Gauss sums}\label{sec:gauss}
We follow the treatment of Gauss sum generating series in
\cite[Section 3]{2019The}, using the functional equation of
Hoffstein and Patterson with order four over $\F_Q$, where $Q=q^2$.
For $N\in\F_Q[T]$, define
\[
 \Psi_{q^2}(N,z)=\sum_{F\in\Atwo}G_{q^2}(N,F)z^{\deg F},
 \qquad C_4(N,j)=\sum_{\deg F=j}G_{q^2}(N,F).
\]
For a finite set $S$ of primes over $\F_Q$, let
$\Psi_{q^2}^{(S)}(N,z)$ denote the sum over moduli coprime to
every prime in $S$. In particular, $\widetilde\Psi_{q^2}(N,z)$
excludes the primes dividing $N$. When numerator exponents are
reduced modulo four, this set of excluded primes is retained.
\subsection{The functional equation and meromorphic continuation}
Following \cite[Section 3.1]{2019The}, we first separate the modulus
degrees. For $0\le i\le3$, put
\[
 \phi_i(N,z)=\frac{1}{1-Q^4z^4}
 \sum_{\deg F\equiv i\ (4)}G_{q^2}(N,F)z^{\deg F}.
\]
These are the degree classes of the series used in the original
first-moment argument. More precisely, with $\pi_\infty=T^{-1}$,
\[
 \phi_i(N,z)=\psi(N,\pi_\infty^{-i},z),\qquad
 F\sim\eta\ \Longleftrightarrow\
 F/\eta\in\bigl(\F_Q((T^{-1}))^\times\bigr)^4.
\]
For monic $F$, this condition is equivalent to $\deg F\equiv i\pmod4$.
All sums here are over monic polynomials, so we use the monic
normalization of \cite[Section 3.1]{2019The}. In particular,
\[
 \Psi_{q^2}(N,z)=(1-Q^4z^4)\sum_{i=0}^3\phi_i(N,z).
\]
Write $a=\deg N$ and let $[b]_4$ be the least nonnegative residue of
$b$ modulo four. For monic $N$, the functional equation is
\begin{align}\label{eq:full-gauss-fe}
(1-Q^5z^4)\phi_i(N,z)
=(Qz)^a\bigg[&-Q^2z(Qz)^{-[a+1-2i]_4}(1-Q^{-1})
 \phi_i(N,Q^{-2}z^{-1})\notag\\
 &-W_{N,i}(Qz)^{-3}(1-Q^4z^4)
 \phi_{[a+1-i]_4}(N,Q^{-2}z^{-1})\bigg].
\end{align}
Here $W_{N,i}$ is the finite field Gauss sum of the character
$\chi_4^{2i-1-a}$ on $\F_Q^\times$, where
$\chi_4(c)=\Omega^{-1}(c^{(Q-1)/4})$. With our additive character,
\[
 W_{N,i}=\begin{cases}
 -1,&2i-1-a\equiv0\pmod4,\\
 \kappa_q^{2i-1-a}q,&2i-1-a\not\equiv0\pmod4.
 \end{cases}
\]
This is the specialization $\ell=4$ of the general-order formula
in~\cite[(4.1)]{david2025nonvanishing}, with its base field size
replaced by $Q=q^2$ and with the monic normalization above.
See also the cubic treatment in~\cite[Proposition 3.2]{2019The}. The trace calculation in
Section~\ref{sec:prelim} evaluates the nontrivial finite field Gauss sums.
We use this equation in two ways: first to determine the polynomial
numerators and hence the residues of $\Psi_{q^2}$, and then to transfer
bounds from $\Re\nu=3/2$ to $\Re\nu=1/2$, where $z=Q^{-\nu}$.

\subsection{The polynomial numerator}
\begin{lemma}[{\cite[(4.1)--(4.2)]{david2025nonvanishing}}]
For $i\in\{0,1,2,3\}$, the generating series satisfies
\begin{equation}\label{eq:rationality}
 \frac{1}{1-Q^4 z^4}\sum_{\deg F\equiv i\,(4)}G_{q^2}(N,F)z^{\deg F}
 =\frac{z^iP_4(N,i,z^4)}{1-Q^{5}z^4},
\end{equation}
where $\deg P_4\le\lfloor(\deg N+1-i)/4\rfloor$.
\end{lemma}
We use \eqref{eq:full-gauss-fe} and this degree bound in
Section~\ref{sec:polynomial-fe}.

\subsection{A finite expression for the generating series}
For an integer $n$, write $[n]_4$ for its least nonnegative residue modulo four. Set $R=\kappa_qQz$ and
\[J_4(y,R)=1+q(1-y)\sum_{i=1}^{3}R^iy^{i-1}-q^2R^4 y^{3}.
\]
Define the coefficients $K_{4,\delta}(R)$ by
\begin{equation}\label{eq:kernel}
 \mathcal K_4(y,R)=\sum_{\delta\ge0}K_{4,\delta}(R)y^\delta
 =\frac{(1-R^4)J_4(y,R)}{(1-q^2R^4)(1-y)(1-R^4 y^4)}.
\end{equation}
For $a=\deg N$,
\begin{equation}\label{eq:halflength}
 \Psi_{q^2}(N,\kappa_qR/Q)=
 \sum_{0\le j\le a/2}\kappa_q^jQ^{-j}C_4(N,j)R^jK_{4,a-2j}(R).
\end{equation}
Section~\ref{sec:polynomial-fe} proves this identity for every nonzero
monic $N\in\F_Q[T]$ from the complete functional equation and the
degree bound in \eqref{eq:rationality}. In particular it is not inferred
from a residue identity or from an average over rational parameters.
Section~\ref{app:kernel} records the corresponding unrestricted average.

For $|R|<1$, $(1-q^2R^4)K_{4,\delta}(R)$ is bounded uniformly
in $\delta$ on compact sets. The functional equation also gives
the coefficient bound
\[\ll (1+\delta)^C|z|^j\left[1+(Q|z|)^\delta+
 (Q^{5/8}|z|^{1/2})^\delta\right].
\]
Writing $|z|=Q^{-\kappa}$, both bases in the powers of
$\delta$ are at most one when $\kappa\ge5/4$. For
$1<\kappa<5/4$, we instead use the uniform bound for
$(1-q^2R^4)K_{4,\delta}(R)$ on compact subsets of $|R|<1$
established directly from \eqref{eq:kernel}. The regions with
$\kappa\le1$ reached later use analytic continuation of the
assembled series, not this termwise bound in $\delta$.

\subsection{The polynomial functional equation}\label{sec:polynomial-fe}
Write $a=\deg N$, $j_i=[a+1-i]_4$ and $r_i=[a+1-2i]_4$.
Set
\[
 H_i(R)=(\kappa_qR/Q)^iP_4(N,i,R^4/Q^4),\qquad
 H(R)=\sum_{i=0}^3H_i(R),\qquad
 B(R)=\frac{1-R^4}{1-q^2R^4}.
\]
Thus $H_i$ contains only powers congruent to $i$ modulo four,
$\deg H\le a+1$, and
\begin{equation}\label{eq:polynomial-normalization}
 \Psi_{q^2}(N,\kappa_qR/Q)=B(R)H(R).
\end{equation}
Substituting in \eqref{eq:full-gauss-fe} gives
\begin{align}\label{eq:polynomial-fe}
 (R^4-q^2)H_i(R)
 & =-(q^2-1)R^{a+5-r_i}H_i(1/R)\notag\\
 &\quad-w_{r_i}R^{a+1}(1-R^4)H_{j_i}(1/R),
 \qquad w_0=-1,\quad w_1=w_2=w_3=q.
\end{align}
Indeed $\kappa_q^{a+1-r_i}=1$ and
$W_{N,i}\kappa_q^{a-3}=w_{r_i}$. Monicity leaves no further
leading-coefficient character.

First we prove uniqueness from the coefficients of $\Psi_{q^2}$ of
degrees at most $\lfloor a/2\rfloor$. Since $B(0)=1$, these determine
the same range of coefficients of $H$. Let $h_t$ denote the coefficient
of $R^t$ in $H$, and put $h_t=0$ outside
$0\le t\le a+1$. For $k>(a+1)/2$, take
$i=[a+1-k]_4$ and compare the coefficient of $R^{a+1-k}$ in
\eqref{eq:polynomial-fe}. The left side uses only known lower
coefficients. On the right $h_k$ has nonzero coefficient $-w_{r_i}$,
and the other indices are $k+4-r_i$ and $k+4$, both greater than $k$.
Descending induction determines all coefficients above the midpoint.
If $a=2m-1$, the remaining coefficient satisfies
\begin{equation}\label{eq:kernel-midpoint}
 (q^2+1)h_m=h_{m-4}+q^2h_{m+4}.
\end{equation}
This is the coefficient of $R^m$ in the self-paired class $i=[m]_4$,
where $r_i=0$. Both coefficients on the right are already known.
For even $a$ there is no integral midpoint. This proves uniqueness.

We next construct the solution whose only nonzero lower coefficient
is its constant coefficient, equal to one. For $\delta=4d+t$,
$0\le t\le3$, define
\begin{equation}\label{eq:kernel-polynomials}
 H_\delta(R)=
 \begin{cases}
 1+(1-q^2)\displaystyle\sum_{v=1}^{d}R^{4v}
       +qR^{4d+t+1},&0\le t\le2,\\[3pt]
 (1-q^2R^4)\displaystyle\sum_{v=0}^{d}R^{4v},&t=3.
 \end{cases}
\end{equation}
These polynomials satisfy \eqref{eq:polynomial-fe} with $a=\delta$.
To check this for $t\le2$, put $p=t+1$, $L=4d+p$ and
$A(R)=1+(1-q^2)\sum_{v=1}^dR^{4v}$. Only classes $0$ and $p$
occur, with components $A(R)$ and $qR^L$. Their equations follow
from the finite geometric identities
\begin{align*}
 &(R^4-q^2)A(R)+(q^2-1)R^{4d+4}A(1/R)+q^2(1-R^4)=0,\\
 &(1-R^4)A(1/R)=q^2-R^4-(q^2-1)R^{-4d}.
\end{align*}
The other classes give zero. For $t=3$, only class zero occurs;
the equation reduces to
\[
 (R^4-q^2)H_\delta(R)
 =R^{4d+4}(1-q^2R^4)H_\delta(1/R),
\]
which follows from
$\sum_{v=0}^dR^{-4v}=R^{-4d}\sum_{v=0}^dR^{4v}$.
Moreover,
\[
 B(R)H_\delta(R)=1+O(R^{\delta+1})\qquad(R\longrightarrow0).
\]
For $t\le2$ this follows by writing the product as
$1-R^{4d+4}+B(R)(q^2R^{4d+4}+qR^{\delta+1})$;
for $t=3$ the product is $1-R^{\delta+1}$.
Summing \eqref{eq:kernel-polynomials} over $\delta$ gives
\[
 \sum_{\delta\ge0}H_\delta(R)y^\delta
 =\frac{J_4(y,R)}{(1-y)(1-R^4y^4)}.
\]
Consequently $K_{4,\delta}=BH_\delta$, with generating function
\eqref{eq:kernel}.

Finally put $c_k=\kappa_q^kQ^{-k}C_4(N,k)$ and form
\[
 H^*(R)=\sum_{0\le k\le a/2}c_kR^kH_{a-2k}(R).
\]
Multiplying each component by $R^k$, shifting its class by $k$,
and replacing $a-2k$ by $a$ preserves \eqref{eq:polynomial-fe}:
$r_i$ is unchanged and every term acquires the same factor $R^k$.
Also $\deg H^*\le a+1$, and
\[
 R^kK_{4,a-2k}(R)=R^k+O(R^{a-k+1}).
\]
Since $a-k+1>\lfloor a/2\rfloor$, the lower coefficients of
$BH^*$ are exactly $c_0,\ldots,c_{\lfloor a/2\rfloor}$.
Uniqueness gives $H^*=H$, proving \eqref{eq:halflength}.
The defining Gauss series converges absolutely for $|z|<Q^{-2}$,
so this equality of germs is an identity of rational functions by
\eqref{eq:rationality}. It therefore holds at every nonsingular
point, not only at the four Gauss poles.

For $|R|\le r<1$, the coefficients of
$((1-y)(1-R^4y^4))^{-1}$ are bounded by $(1-r^4)^{-1}$.
The polynomial numerator in \eqref{eq:kernel} therefore proves
the stated uniform bound for $(1-q^2R^4)K_{4,\delta}$.
This argument concerns the unrestricted series; exclusions are
imposed separately by \eqref{eq:hecke}.

\subsection{Residues}\label{sec:gauss-residue}
For $\zeta^4=1$, put
\[
 z_{4,\zeta}=\kappa_q\zeta Q^{-5/4},\qquad
 \Theta_{4,\zeta}(N)=\sum_{i=0}^{3}z_{4,\zeta}^{i}
 P_4(N,i,Q^{-5}).
\]
Summing \eqref{eq:rationality} over the four degree classes gives
\[
 \Psi_{q^2}(N,z)=\frac{1-Q^4z^4}{1-Q^5z^4}
 \sum_{i=0}^{3}z^iP_4(N,i,z^4).
\]
Consequently,
\begin{align}\label{eq:gauss-residue-value}
 \left[(1-Q^5z^4)\Psi_{q^2}(N,z)\right]_{z=z_{4,\zeta}}
 &=(1-q^{-2})\Theta_{4,\zeta}(N),\\
 \Res_{z=z_{4,\zeta}}\Psi_{q^2}(N,z)
 &=-\frac{z_{4,\zeta}}4(1-q^{-2})\Theta_{4,\zeta}(N).
 \nonumber
\end{align}
The functional equation gives a finite formula for these residues.
Let $R_\zeta=\zeta q^{-1/2}$ and define $b_{\zeta,\delta}$ by
\[
 \sum_{\delta\ge0}b_{\zeta,\delta}Y^\delta
 =\frac{J_4(Y,R_\zeta)}{(1-Y)(1-q^{-2}Y^4)},\qquad
 K_\zeta(a,j)=\kappa_q^jQ^{-j}R_\zeta^j b_{\zeta,a-2j}.
\]
Multiplying \eqref{eq:halflength} by $1-Q^5z^4$ and taking
$z=z_{4,\zeta}$ yields, for $a=\deg N$,
\begin{equation}\label{eq:gauss-residue-finite}
 \Theta_{4,\zeta}(N)=
 \sum_{0\le j\le a/2}K_\zeta(a,j)C_4(N,j).
\end{equation}
We use this finite expression to average the residues over
numerators in $\F_q[T]$ in Section~\ref{sec:residues}.

\subsection{An average of the generating series}\label{app:kernel}
The corresponding unrestricted average over rational polynomials $N$ is
\[\sum_N x^{\deg N}\Psi_{q^2}(N,z)=
 \frac{(1-R^4)(1-Ry^2/q)J_4(y,R)}
 {(1-q^2R^4)(1-y)(1-Ry^2)(1-R^4 y^4)},\qquad y=qx.
\]
To obtain this identity, take $H=1$ in the complete average.
For $b\ge2j$, the trace condition forces a modulus contributing a
nonzero character value to lie in $\F_q[T]$. Hence
\[
 c_1(j)=\kappa_q^j\sum_{F\in\mathcal A_q,\ \deg F=j}\phi_q(F),
 \qquad \sum_{F\in\mathcal A_q}\phi_q(F)t^{\deg F}
 =\frac{1-qt}{1-q^2t}.
\]
In the finite expression, write $a=2j+\delta$. Summing first over
$j$ gives $(1-Ry^2/q)/(1-Ry^2)$, and summing over $\delta$
gives $\mathcal K_4(y,R)$, proving the displayed formula.
Let $V_a$ denote the coefficient of $y^a$ in this rational function. After multiplying by
$(1-y^{8})(1-R^4 y^{8})(1-R^{8}y^{8})$, the numerator degree is at most $23$. Thus its coefficient sequence obeys
\[V_a=(1+R^4+R^{8})V_{a-8}
 -(R^4+R^{8}+R^{12})V_{a-16}
 +R^{12}V_{a-24},\qquad a\ge24.
\]

\subsection{An upper bound}
We now apply the functional equation to bound the generating series,
following the argument of \cite[Lemma 3.10]{2019The}.
Write $z=Q^{-\nu}$ and $\theta=\Re\nu$.
\begin{lemma}[{\cite[Theorem 4.1 and Corollary 5.4]{david2025nonvanishing}}]\label{upper2}
For $1/2\le\theta\le3/2$, on compact subsets avoiding the possible poles,
\[
 \Psi_{q^2}(N,z)\ll_{q,\eps}|N|_2^{\frac12(\frac32-\theta)+\eps}.
\]
The residues satisfy
\[
 \Theta_{4,\zeta}(N)\ll_{q,\eps}|N|_2^{1/8+\eps}.
\]
For $1\le\theta<3/2$, the series with coprimality imposed satisfies
\[
 \widetilde\Psi_{q^2}(N,z)
 \ll_{q,\eps}|N|_2^{\frac12(\frac32-\theta)+\eps}
\]
on compact subsets avoiding its possible poles.
\end{lemma}
\begin{proof}
Put
\[
 \psi_i(N,z)=\frac{1}{1-Q^4z^4}
 \sum_{\deg F\equiv i\ (4)}G_{q^2}(N,F)z^{\deg F}.
\]
The finite polynomial formula and the local Gauss sum bounds give
$\psi_i(N,Q^{-\nu})\ll |N|_2^\eps$ on $\theta=3/2$.
The functional equation relates this line to $\theta=1/2$ and
gives the bound $|N|_2^{1/2+\eps}$ there. Multiplication by
$1-Q^{5-4\nu}$ removes the poles. Interpolating between the
two lines and summing the four degree classes gives the first
estimate. The same bound for the pole-free function at
$\theta=5/4$, together with \eqref{eq:gauss-residue-value}, gives
the estimate for $\Theta_{4,\zeta}$.

For the coprime series, take order four and $a=1$ in
\cite[Corollary 5.4]{david2025nonvanishing}. In the range
$1\le\theta<3/2$, the exponent
$4(3/4-\theta/2)-3/2=3/2-2\theta$ is negative.
The finite divisor sums therefore contribute at most
$|N|_2^\eps$, which gives the stated bound.
\end{proof}
The estimates below use the finite expression \eqref{eq:halflength}
before taking absolute values, so that cancellation in the averages
over polynomials in $\F_q[T]$ is preserved.

\section{The double Dirichlet series and its main residue}\label{sec:generating}\subsection{The generating function}
  For $\Re(s)$ and $\Re(w)$ sufficiently large, we define
 \begin{equation}
 \label{eq:A4-definition}
     A_4(s, w)=\sum_{\substack{F\in\mathcal{H}_{q^2} \\ (F,F^{\sigma_0})=1}} \frac{L_q(w, \chi_{F})}{|F|^s_2}.
 \end{equation}

The term $F=1$ contributes $\zeta_q(v)=(1-qv)^{-1}$, where $v=q^{-w}$. We retain $A_4$ for the Euler product calculations and set
\begin{equation}\label{A-positive}
 \widehat A_4(u,v)=A_4(u,v)-\frac{1}{1-qv}
 =\sum_{\substack{F\in\mathcal H_{q^2}\\\deg F\ge1,\ (F,F^{\sigma_0})=1}}
 L_q(v,\chi_F)u^{\deg F}.
\end{equation}
In additive coordinates, $\widehat A_4(s,w)$ denotes the same function with $u=q^{-2s}$ and $v=q^{-w}$. The subtraction leaves every positive-degree coefficient in $u$, and the residue at $u=q^{-2}$, unchanged. It removes the extraneous pole at $v=q^{-1}$ from the continuation argument.

For $m=g/3+1$ and sufficiently small $r>0$, Cauchy's coefficient formula gives
\begin{equation}\label{eq:coefficient}
 \sum_{\substack{F\in\Htwo\\(F,F^{\sigma_0})=1}}
 L_q(1/2,\chi_F)
 =\frac{1}{2\pi i}\oint_{|u|=r}
 \widehat A_4(u,q^{-1/2})\,\frac{\dd u}{u^{m+1}}.
\end{equation}
We use additive coordinates $(s,w)$ for the tube theorem and
multiplicative coordinates $(u,v)=(q^{-2s},q^{-w})$ for coefficient
extraction.

\subsection{The first convergence region}
 To analyze the behavior of $A_4(u,v)$, we use M\"obius inversion over $\F_q[T]$ 
 \begin{equation}     \sum_{\substack{D\in\Aq\\ D|F}}\mu(D)=\left\{ \begin{array}{ll}
         1 & \mbox{$F$ has no prime divisor in } \mathbb{F}_q[T]; \\
         0 & \mbox{otherwise}.
     \end{array}
     \right.
 \end{equation}
 This allows us to remove the divisor condition.
 Applying this, $A_4(s,w)$ can be rewritten as
  \begin{align}
  A_4(s, w)=&\sum_{\substack{F\in\Hq \\ (F,F^{\sigma_0})=1}}\frac{L_q(w,   \chi_{F})}{|F|_2^s}\notag\\
  =&\sum_{D\in\Aq} \mu(D)\sum_{\substack{F\in\mathcal{H}_{q^2}\\ (D,F)=1}} \frac{L_q(w, \chi_{DF})}{|DF|_2^s}\notag\\
  =&\sum_{D\in\Aq}\mu(D)\sum_{\substack{F\in\mathcal{H}_{q^2}\\ (D,F)=1}}\sum_{N\in\Aq}\frac{\chi_{DF}(N)}{|N|^{w}|DF|_2^s}\notag\\
  =&\sum_{N\in\Aq}\frac{1}{|N|^w}\sum_{\substack{D\in\Aq\\ }}\frac{\mu(D)\chi_D(N)}{|D|_2^s}\sum_{\substack{F\in\Hq\\ (D,F)=1}}\frac{\chi_F(N)}{|F|_2^s}\notag\\
  =&\sum_{N\in\Aq}\frac{1}{|N|^w}\frac{L_{q^2}(s, \chi^{(N)})}{L_{q^2}(2s, \chi^{2(N)})}\sum_{\substack{D\in\Aq\\ }}\frac{\mu(D)\chi_D(N)}{|D|_2^s}\prod_{\substack{P_2|D\\ P_2\in\Atwo}}\left(1+\frac{\chi_{P_2}(N)}{|P_2|_2^s}\right)^{-1}.
 \end{align}
 The last equality follows from
 \begin{align*}
     \sum_{\substack{F\in\Hq\\ (D,F)=1}}\frac{\chi_F(N)}{|F|_2^s}&=\prod_{\substack{P_2\in\Atwo\\ P_2\nmid D}}\left(1+\frac{\chi_{P_2}(N)}{|P_2|^s_2}\right)\\ & =\prod_{\substack{P_2\in\Atwo}}\left(1+\frac{\chi_{P_2}(N)}{|P_2|^s_2}\right)\prod_{\substack{P_2\in\Atwo\\ P_2| D}}\left(1+\frac{\chi_{P_2}(N)}{|P_2|^s_2}\right)^{-1}\\
     & =\frac{L_{q^2}(s,\chi^{(N)})}{L_{q^2}(2s,\chi^{2(N)})}\prod_{\substack{P_2\in\Atwo\\ P_2| D}}\left(1+\frac{\chi_{P_2}(N)}{|P_2|^s_2}\right)^{-1}.
 \end{align*}
 Here $\chi^{(N)}(F)=\chi_F(N)$ is a multiplicative character induced by $\chi_F$. $L_{q^2}(s,\chi^{(N)})$ denotes the $L$-function over $\mathbb{F}_{q^2}(T)$. %We recall that $|F|=q^{\deg F}$ and $|F|_2=q^{2\deg F}$.

 The following identity simplifies the sum over $D$.
 \begin{lemma}
     If $D, F\in\Aq$ and $(D, F)=1$, then we have $\chi_D(F)=1$.
 \end{lemma}
 \begin{proof}
     By multiplicativity, it suffices to treat an irreducible $D$. If $2|\deg D$, then $D=\pi\pi^{\sigma_0}$ for some prime $\pi\in\Atwo$, where $\sigma_0$ is the generator of Gal$(\mathbb{F}_{q^2}/\mathbb{F}_{q})$. We have $\chi_{\pi^{\sigma_0}}|_{\Aq}=\overline{\chi_{\pi}|_{\Aq}}$. Hence, we must have $\chi_D(F)=\chi_{\pi}(F)\chi_{\pi^{\sigma_0}}(F)=1$ for $F\in\Aq$ and $(F,D)=1$. If $d=\deg D$ is odd, then $D$ remains irreducible over $\mathbb F_{q^2}$. The residue of $F$ lies in $\mathbb F_{q^d}^{\times}$, and
\[
 \frac{q^{2d}-1}{4}=(q^d-1)\frac{q^d+1}{4}
\]
is a multiple of $q^d-1$, since $q^d\equiv3\pmod4$. Thus $\chi_D(F)=1$ in this case as well. Multiplicativity proves the assertion for arbitrary $D$.\qedhere
     
 \end{proof}
 Using the lemma, we obtain
\begingroup\small
\begin{align*}
    A_4(s, w)=&\sum_{N\in\Aq}\frac{1}{|N|^{w}}\frac{L_{q^2}(s, \chi^{(N)})}{L_{q^2}(2s, \chi^{2(N)})}\sum_{\substack{D\in\Aq\\ (D,N)=1}}\frac{\mu(D)}{|D|^s_2}\prod_{\substack{P_2\in\Atwo\\ P_2|D}}\left(1+\frac{\chi_{P_2}(N)}{|P_2|^s_2}\right)^{-1}\\
    =&\sum_{N\in\Aq}\frac{1}{|N|^{w}}\frac{L_{q^2}(s, \chi^{(N)})}{L_{q^2}(2s, \chi^{2(N)})}\prod_{\substack{P_1\in\Aq\\ (P_1,N)=1}}\left(1-\frac{1}{|P_1|^s_2}\prod_{\substack{P_2|P_1\\
    P_2\in\Atwo}}\left(1+\frac{\chi_{P_2}(N)}{|P_2|_2^s}\right)^{-1}\right) \\
    =&\sum_{N\in\Aq}\frac{1}{|N|^{w}}\frac{L_{q^2}(s, \chi^{(N)})}{L_{q^2}(2s, \chi^{2(N)})}P(s,\chi^{(N)})\prod_{\substack{P_1|N\\ P_1\in\Aq}}\left(1-\frac{1}{|P_1|_2^s}\right)^{-1},
\end{align*}\endgroup
where
\begin{equation}
    P(s, \chi^{(N)})=\prod_{P_1\in\Aq}\left(1-\frac{1}{|P_1|_2^s}\prod_{\substack{P_2|P_1\\ P_2\in\Atwo}}\left(1+\frac{\chi_{P_2}(N)}{|P_2|_2^s}\right)^{-1}\right).
\end{equation}

Write $\sigma=\Re(s)>\frac12$. Uniformly in $N$, the local factors give 
\begin{align}\label{P1}
P(s,\chi^{(N)})=&\prod_{P_1\in\Aq}\left(1-\frac{1}{|P_1|_2^s}+O_q\left(|P_1|_2^{-3\sigma/2}\right)\right)\notag\\
=&\zeta_q^{-1}(2s)\prod\limits_{P_1\in\Aq}\left(1+O_q\left(|P_1|_2^{-3\sigma/2}\right)\right),
\end{align}
where $\zeta_q(s)=\frac{1}{1-q^{1-s}}$ is the zeta function for field $\mathbb{F}_q(T)$. 
For an even-degree prime $P_1$, its two factors over $\mathbb F_{q^2}$ have half its degree; this accounts for the exponent $3\sigma/2$ in \eqref{P1}. The error product converges absolutely in the stated region. On each compact subset $K$ of $\Re(s)>1/2$, the finite Euler factors satisfy
\begin{align}\label{P2}
 \prod_{\substack{P_1\mid N\\P_1\in\mathcal A_q}}
 \left|1-|P_1|_2^{-s}\right|^{-1}
 \ll_{K,q,\varepsilon}|N|_2^\varepsilon.
\end{align}
Indeed, factors at sufficiently large primes are bounded by an arbitrarily small power of the prime norm, and the remaining finitely many primes contribute a constant.

Taking $u=q^{-2s}$ and $v=q^{-w}$, we have 
\begin{align}\label{A}
    A_4(u,v)=\sum_{N\in\Aq}v^{\deg N}\frac{L_{q^2}(u, \chi^{(N)})}{L_{q^2}(u^2, \chi^{2(N)})}P(u,\chi^{(N)})\prod_{\substack{P_1|N\\ P_1\in\Aq}}\left(1-u^{\deg P_1}\right)^{-1},
\end{align}
where $$L_{q^2}(u,\chi^{(N)})=\sum\limits_{F\in\Atwo}\chi_F(N)u^{\deg F}$$ 
and 
$$P(u, \chi^{(N)})=\prod\limits_{P_1\in\Aq}(1-u^{\deg P_1}\prod\limits_{_{P_2|P_1}}(1+\chi_{P_2}(N)u^{\deg P_2})^{-1}).$$ It is easy to see that $P(u, \chi^{(N)})$ is holomorphic for $|u|<q^{-1}$. 

Let $K$ be a compact subset of $|u|<q^{-1}$. Quartic reciprocity identifies $\chi^{(N)}$ with the quartic residue character modulo $N$, so it is principal precisely when $N$ is a fourth power. If $N$ is not a fourth power, passage to the inducing primitive character, Lemma~\ref{LLH}, and the finite Euler factor estimate used in \eqref{P2} give
\begin{align}\label{L}
 \sup_{u\in K}|L_{q^2}(u,\chi^{(N)})|
 \ll_{K,q,\varepsilon}q^{\varepsilon\deg N}.
\end{align}
The same argument includes the quadratic characters arising when $N$ is a square but not a fourth power. The conductor of the inducing character divides $N$; the omitted Euler factors contribute at most $q^{\varepsilon\deg N}$, after relabeling $\varepsilon$.

If $N=M^4$, then
\[
 L_{q^2}(u,\chi^{(N)})
 =\frac{1}{1-q^2u}\prod_{\substack{\pi\mid M\\\pi\in\mathcal A_{q^2}}}
 (1-u^{\deg\pi}).
\]
Consequently $(u-q^{-2})L_{q^2}(u,\chi^{(N)})$ is holomorphic and is also bounded by $O_{K,q,\varepsilon}(q^{\varepsilon\deg N})$. For every $N$, the inverse of $L_{q^2}(u^2,\chi^{2(N)})$ is represented by an absolutely convergent Euler product on $|u|<q^{-1}$ and is uniformly bounded on $K$. The product $P(u,\chi^{(N)})$ is likewise uniformly bounded there, since its local factors are $1+O_K(|u|^{\deg P_1})$.

Applying these bounds to \eqref{A}, with $\varepsilon$ relabeled, yields
\begin{align}\label{S1-regularized-bound}
 |(u-q^{-2})A_4(u,v)|
 \ll_{K,q,\varepsilon}\sum_{N\in\mathcal A_q}
 (q^\varepsilon|v|)^{\deg N},\qquad u\in K.
\end{align}
Choosing $\varepsilon$ small on each compact subset proves normal convergence of the regularized series and hence holomorphy of $(u-q^{-2})A_4(u,v)$ on
\begin{align}
 S_1=\{(u,v):|u|<q^{-1},\ |v|<q^{-1}\}.
\end{align}

  \subsection{Residue at \texorpdfstring{$u=q^{-2}$}{u=q to the power -2}}\label{PZ}
  We now analyze the possible poles at $u=q^{-2}$. Let $\zeta_{q^2}(u)=\frac{1}{1-q^2u}$ be the zeta function for $\mathbb{F}_{q^2}(T)$ in the variable $u$. The possible pole comes from the terms for which $N$ is a fourth power. We have
  \begin{align}     
  &\Res_{u=q^{-2}}\frac{A_4(u,v)}{u^{g/3+2}}\notag\\
  =&\Res_{u=q^{-2}}\frac{1}{u^{g/3+2}}\sum_{N\in\Aq}v^{4\deg N}\frac{\zeta_{q^2}(u)}{\zeta_{q^2}(u^2)}\prod_{\substack{P_2\in\Atwo\\P_2|N}}\frac{1-u^{\deg P_2}}{1-u^{2\deg P_2}}\prod_{P_1 \not| N}(1-u^{\deg P_1}\prod_{\substack{P_2\in\Atwo\\ P_2|P_1}}(1+u^{\deg P_2})^{-1})\notag\\
  =&\Res_{u=q^{-2}}\frac{\zeta_{q^2}(u)}{u^{g/3+2}\zeta_{q^2}(u^2)}P(u)Z(u,v),
  \end{align}
where 
\begin{equation}
    P(u)=\prod_{P_1\in\Aq}(1-u^{\deg P_1}\prod_{\substack{P_2\in\Atwo\\ P_2|P_1}}(1+u^{\deg P_2})^{-1})
\end{equation}
and
\begin{equation}
    Z(u,v)=\sum_{N\in\Aq}v^{4\deg N}\prod_{\substack{P_2\in\Atwo\\P_2|N}}(1+u^{\deg P_2})^{-1}\prod_{\substack{P_1\in\Aq\\ P_1|N}}(1-u^{\deg P_1}\prod_{\substack{P_2\in\Atwo\\ P_2|P_1}}(1+u^{\deg P_2})^{-1})^{-1}.
\end{equation}
  Taking $v=q^{-\frac{1}{2}}$, we then have
  \begin{flalign}
  \Res_{u=q^{-2}}\frac{A_4(u,q^{-\frac{1}{2}})}{u^{g/3+2}}=q^{\frac{2g}{3}}(1-q^{2})P(q^{-2})Z(q^{-2},q^{-\frac{1}{2}}).
  \end{flalign}

We also record that the leading constant is strictly positive. At $u=q^{-2}$, the local factor of $P(u)$ corresponding to a prime $P\in\mathcal A_q$ of degree $d$ is
\[
 p_P=\begin{cases}
 (1+q^{-2d})^{-1},&d\text{ odd},\\[2pt]
 \dfrac{1+2q^{-d}}{(1+q^{-d})^2},&d\text{ even}.
 \end{cases}
\]
These factors are positive and satisfy $p_P=1+O(q^{-2d})$. Since $\sum_Pq^{-2\deg P}<\infty$, their product converges to a strictly positive number. In the defining series for $Z(q^{-2},q^{-1/2})$, the factor associated with each distinct prime divisor $P\mid N$ is
\[
 \gamma_P=\begin{cases}1,&d\text{ odd},\\(1+2q^{-d})^{-1},&d\text{ even}.\end{cases}
\]
Thus $0<\gamma_P\le1$, and
\[
 1\le Z(q^{-2},q^{-1/2})
 =\sum_{N\in\mathcal A_q}q^{-2\deg N}\prod_{P\mid N}\gamma_P
 \le\sum_{N\in\mathcal A_q}q^{-2\deg N}
 =\frac{1}{1-q^{-1}}<\infty.
\]
Consequently,
\begin{equation}\label{positive-main-constant}
 C_q=(q^2-1)P(q^{-2})Z(q^{-2},q^{-1/2})>0.
\end{equation}

The same calculation gives, for $w$ near $1/2$,
\begin{equation}\label{eq:Zproduct}
 Z(q^{-2},q^{-w})=\zeta_q(4w)
 \prod_{\deg P\ \text{even}}
 \left(1-\frac{2|P|^{-4w}}{|P|+2}\right).
\end{equation}
The coefficient of $(1-q^2u)^{-1}$ in $A_4(u,q^{-w})$ is
$(1-q^{-2})P(q^{-2})Z(q^{-2},q^{-w})$.
To express the counting factor directly, use the Euler product over
$\F_{q^2}[T]$:
\[
 1-q^{-2}
 =\prod_{\deg P\ \mathrm{odd}}(1-p^{-4})
  \prod_{\deg P\ \mathrm{even}}(1-p^{-2})^2.
\]
Multiplying by the local factors of $P(q^{-2})$ gives
\begin{equation}\label{eq:counting-euler-product}
 (1-q^{-2})P(q^{-2})
 =\prod_{\deg P\ \mathrm{odd}}(1-p^{-2})
  \prod_{\deg P\ \mathrm{even}}(1+2p^{-1})(1-p^{-1})^2.
\end{equation}
Combining this identity with \eqref{eq:Zproduct} at $w=1/2$
and $\zeta_q(2)=(1-q^{-1})^{-1}$ proves
\eqref{eq:main-euler-product}. Each local factor is positive
and differs from one by $O(p^{-2})$.
The coefficient of $(1-q^2u)^{-1}$ is $q^{-2}C_q$ at the
central point, so its contribution to the coefficient of $u^m$
is $C_q q^{2m-2}=C_q q^{2g/3}$.

For later use, write $p=|P|$ and set
\begin{equation}\label{eq:gamma}
 \gamma(N)=\prod_{\substack{P\mid N\\\deg P\ \text{even}}}(1+2/p)^{-1}.
\end{equation}
\subsection{The second region}
Write $v=q^{-w}$ and use $L_q(v,\chi)$ for the polynomial in $v$ associated with $L_q(w,\chi)$. For $\deg F\ge1$, Lemma~\ref{LLH} gives
\[
 |L_q(v,\chi_F)|\ll_{q,\varepsilon}|F|_2^\varepsilon
 \qquad (|v|\le q^{-1/2}).
\]
Consequently, the defining series for $\widehat A_4(u,v)$ converges normally on compact subsets satisfying $|u|<q^{-2}$ and $|v|\le q^{-1/2}$.
For $|v|>q^{-1/2}$, Lemma~\ref{fe} gives
\begin{align*}
 \widehat A_4(u,v)
 &=\frac{1}{qv^2}\frac{1-v}{1-1/(qv)}
 \sum_{\substack{F\in\mathcal H_{q^2}\\\deg F\ge1,\ (F,F^{\sigma_0})=1}}
 \omega(\chi_F)L_q(1/(qv),\overline{\chi_F})(quv^2)^{\deg F}.
\end{align*}
Since $|\omega(\chi_F)|=1$, on each compact set $K$ in this range,
\[
 |L_q(v,\chi_F)|\ll_{K,q,\varepsilon}
 (q^{1+\varepsilon}|v|^2)^{\deg F}.
\]
There are at most $q^{2n}$ polynomials $F$ of degree $n$, so the latter series converges normally when $|u|^{1/2}|v|<q^{-3/2}$. Combining the two estimates (and choosing $\varepsilon$ small relative to the strict inequalities) gives holomorphy of $\widehat A_4$ on the open region
\begin{align}
 S_2=\{(u,v): |u|<q^{-2},\ |u|^{1/2}|v|<q^{-3/2}\}.
\end{align}

\section{M\"obius inversion and averages of Gauss sums}\label{sec:D}
We derive the expression to be estimated in
Section~\ref{sec:continuation}. We first retain the M\"obius weights in
the sum over $D$ and impose the coprimality conditions on the modulus $F$.
We then evaluate the average over the coefficients of the rational
numerator. The resulting identities and local character sums are the
inputs for the upper bounds in the next section.

\subsection{M\"obius inversion and the functional equation}
In this section $N,D\in\mathcal A_q$. We write
\[
 \sigma=\Re s,\qquad\tau=\Re w,\qquad
 k=s+w,\qquad\kappa=\Re k,\qquad x=q^{w-1}.
\]
Let $\mu$ denote the polynomial M\"obius function on $\F_q[T]$:
$\mu(D)=0$ if $D$ is divisible by the square of a nonconstant polynomial, and
$\mu(D)=(-1)^r$ if $D$ is a product of $r$ distinct monic
irreducible polynomials. M\"obius inversion over $\F_q[T]$
removes prime divisors over $\F_q$ from the extension modulus.
Applying the functional equation to the resulting $L$-series gives,
initially for $\Re w<0$, $\Re(s+w)>3/2$,
\begingroup\small
\begin{align}\label{eq:dual}
 \widehat A_4(s,w)&=\mathcal P(w)\left[
 \sum_N x^{\deg N}\sum_{\substack{D\in\F_q[T]\text{ squarefree}\\(D,N)=1}}
 \mu(D)\norm D^{1-2(s+w)}
 \widetilde\Psi_{q^2}(ND^{2},\kappa_qQ^{-s-w})
 -\zeta_q(1-w)\right],\\
 x&=q^{w-1},\qquad
 \mathcal P(w)=q^{2w-1}\frac{1-q^{-w}}{1-q^{w-1}}.\nonumber
\end{align}\endgroup
Lemma~\ref{quartic-restriction} gives
$G_{q^2}(N,D)=\kappa_q^{\deg D}q^{\deg D}$ for $(N,D)=1$.
This phase cancels the sign from the functional equation, and the
numerator in the Gauss series is $ND^2$.

The factors used to form a pole-free function are
\begin{equation}\label{eq:clearing}
 C(s)=1-q^{2-2s},\qquad
 G_4(s,w)=1-q^{10-8(s+w)},\qquad
 N_4(s,w)=1-q^{8-8s-4w}.
\end{equation}

\subsection{Excluding prime divisors from the modulus}
Fix the set $S$ of excluded primes before reducing the numerator
exponents modulo $4$. In \eqref{eq:dual}, $S$ consists of every
extension prime dividing the original $ND$. If $ND^2=fK^4$ with
$f$ fourth-power-free, then
$\Psi_{q^2}^{(S)}(ND^2,z)=\Psi_{q^2}^{(S)}(f,z)$ coefficientwise,
because every allowed modulus is coprime to $K$. This equality
does not change $x^{\deg N}\mu(D)|D|^{1-2(s+w)}$ or remove a prime
from $S$. In particular the original fourth-power degrees are
summed with their original weights. For the reduced valuations
$a_\pi\in\{0,\ldots,3\}$, put $S_0=\{\pi\in S:a_\pi\le2\}$
and $D_I=\prod_{\pi\in I}\pi^{a_\pi+1}$. Then
\begin{equation}\label{eq:hecke}
 \Psi_{q^2}^{(S)}(f,z)=
 \frac{\displaystyle\sum_{I\subseteq S_0}(-1)^{|I|}G_{q^2}(f,D_I)z^{\deg D_I}
 \Psi_{q^2}\left(f\prod_{\pi\in I}\pi^{2-2a_\pi},z\right)}
 {\displaystyle\prod_{\pi\in S}(1-\norm{\pi}_Q^{3}z^{4\deg\pi})}.
\end{equation}
The Gauss product in \eqref{eq:hecke} includes the factors relating
distinct primes. We use \eqref{eq:hecke} to impose the exclusions
before applying the functional equation to the unrestricted series.

\subsection{Notation for the averages}
For a modulus $F$ of degree $j$ in $G_{q^2}(V,F)$, write $F=AB$, where
$A=(F,F^{\sigma_0})$ and $(B,B^{\sigma_0})=1$.
Then $A\in\F_q[T]$. The parameter $H_R\in\F_q[T]$ denotes
the rational part of the numerator after reduction of its
exponents modulo four in \eqref{eq:hecke}; write $h_R=\deg H_R$.
The selected primes are those chosen in the subset $I$ in that
identity. Selecting one prime $\pi$ above $P$ with reduced numerator valuation
$a_\pi\in\{0,1,2\}$ introduces $\chi_\pi^{-(a_\pi+1)}$ in the
rational numerator variable. The exponent two is the critical
$a_\pi=1$ case summed with signs below; exponents one and three
are retained as fixed or absolutely summed twisted data. The letter $E$
denotes the monic conductor over $\F_q[T]$ of any remaining
multiplicative character in that variable. The weights outside
the Gauss series are precisely $x^{\deg N}$ and
$\mu(D)|D|^{1-2(s+w)}$ from \eqref{eq:dual}.

\subsection{Orthogonality of additive characters}\label{app:fibres}
Let $\psi_q$ be the nontrivial additive character of $\F_q$
used to define $e_q$. We use the standard orthogonality relation
\[
 q^{-1}\sum_{a\in\F_q}\psi_q(a\lambda)
 =\begin{cases}1,&\lambda=0,\\0,&\lambda\ne0,\end{cases}
 \qquad \lambda\in\F_q.
\]
The trace occurs because $e_{q^2}$ is defined using the field trace
from $\F_Q$ to $\F_q$. Thus averaging each free coefficient of a
rational polynomial imposes one trace condition.
For a Laurent series $f$, write $[T^j]f$ for its coefficient of $T^j$. For a polynomial $H\in\F_q[T]$, put $h=\deg H$. If $M=T^b+\sum_{i<b}m_iT^i$, orthogonality in its $b$ free coefficients gives
\begin{align*}
 q^{-b}\sum_{\deg M=b}C_4(HM,j)
 &=\sum_{\deg F=j}\sum_{X\bmod F}\chi_F(X)e_{q^2}(HT^bX/F)\\
 &\quad\times\prod_{i=0}^{b-1}
 \mathbf1_{\operatorname{Tr}_{Q/q}[T^{-i-1}](HX/F)=0}.
\end{align*}
Write $F=AB$, $A=(F,F^{\sigma_0})$ and $D_*=ABB^{\sigma_0}$. Then $A$ is rational, $(B,B^{\sigma_0})=1$, $c=\deg A$, and $\deg D_*=2j-c$. With $Z=XB^{\sigma_0}+X^{\sigma_0}B$, the conditions say the proper part of $HZ/D_*$ is $V/D_*$ with
\[\deg V\le2j-c-b-1=h-\delta-c-1,\qquad\delta=h+b-2j.
\]
The additive phase of the monic term is determined by the coefficient of the highest permitted degree. If $b\ge2j$, only $V=0$ occurs, so the average is independent of $b$.

\subsection{Local character sums subject to the trace congruence}\label{sec:local-character-sums}
With $F=AB$ and $D_*=ABB^{\sigma_0}$ as above, define, for a rational
residue class $v$ modulo $D_*$,
\begin{equation}\label{eq:trace-character-sum}
 \mathcal T_{F,H}(v)=
 \sum_{\substack{Y\bmod F\\
 H(YB^{\sigma_0}+Y^{\sigma_0}B)\equiv v\ (\bmod D_*)}}
 \chi_F(Y).
\end{equation}
This is the character sum over residue classes satisfying the displayed
trace congruence. The Chinese Remainder Theorem reduces its evaluation
to congruences at prime powers. We compute these local character sums
for a rational polynomial $H$ not divisible by a fourth power. At a split prime $P=\pi\pi^{\sigma_0}$, let $p=|P|$, $d=v_P(H)<4$, and let the exponents of $\pi$ and $\pi^{\sigma_0}$ in $F$ be $(e+t,e)$, $t\ge0$. After identifying the conjugate residue fields, the congruence and character weight have the form
\[P^d\eta(\gamma x+P^t\beta y)=V\pmod{P^{e+t}},\qquad
 \chi(x)^{e+t}\chi(y)^{-e},
\]
with units $\eta,\beta,\gamma$. Explicitly, the local character sum is
\[
 \sum_{\substack{x\bmod P^{e+t},\ y\bmod P^e\\
 P^d\eta(\gamma x+P^t\beta y)\equiv V\ (\bmod P^{e+t})}}
 \chi(x)^{e+t}\chi(y)^{-e}.
\]
When $e=0$, the $y$ variable and its character factor are absent.
Characters are zero on nonunits, even when their exponent is zero
modulo $4$.
\begin{itemize}
\item If $t>0$ and $e+t>d$, the local character sum vanishes unless $V=P^dw$ with $w$ a unit and either $e=0$ or $4\mid e$. Its nonzero value is
\[p^d\chi(w/(\eta\gamma))^{e+t}
 \begin{cases}1,&e=0,\\p^{e-1}(p-1),&e>0,\ 4\mid e.\end{cases}
\]
\item If $t=0$ and $e>d$, the local character sum is
\[p^{e+d-1}\chi(\beta/\gamma)^eJ_e(w/\eta),\qquad
 J_e(a)=\sum_{u\in\F_p}\chi(a-u)^e\chi(u)^{-e}.
\]
Here $J_e(0)=p-1$, and for $a\ne0$ the value is $-1$ if $4\nmid e$ and $p-2$ otherwise.
\item If $0<e+t\le d<4$, the complete character sum vanishes. Thus, for a nonzero contribution, the common divisor contains the full factor $P^d$ of $H$.
\end{itemize}
The first formula follows by solving for $x$ modulo $P^{e+t-d}$ and summing the remaining character in $y$. The second follows by scaling the nonzero trace value and evaluating the resulting Jacobi sum.

At an inert prime of odd degree $n$, with $p=q^n$, the residue field is $\F_{p^2}$ and the character is trivial on $\F_p^\times$. A unit $\xi$ of trace zero has character value $\kappa_q^n$. For $e>d$, $V=P^dw$, the local character sum is
\[\chi_P(B)^e p^{e+d-1}J_e(w/\eta),\qquad
 J_e(a)=\sum_{\operatorname{Tr}y=a}\chi_P(y)^e,
\]
where
\[J_e(0)=(p-1)\kappa_q^{ne},\qquad
 J_e(a)=\begin{cases}-\kappa_q^{ne},&a\ne0,\ 4\nmid e,\\p,&a\ne0,\ 4\mid e.\end{cases}
\]
If $0<e\le d<4$, the local character sum is zero. These formulas follow by separating the line of trace zero from the $p-1$ lines of nonzero trace and then lifting each residue solution.

Consequently, with $G=(H,D_*)$, $H=GH_1$ and $V=GW$, the nonzero global character sum has the form
\[\kappa_q^{\deg B}\chi_{A_*W/H_1}(B)
 \times\text{local factors depending only on valuations},
\]
of magnitude at most $q^{c+g}$, $g=\deg G$. In the character, $A_*$ omits prime powers whose exponents are divisible by $4$; their contributions are included in the local factors above. Also $(H_1,D_*)=1$ and
\[\deg W\le h-g-\delta-c-1.
\]
In the untwisted correction this character is nonprincipal: principality would force $H_1\mid W$, contradicting the degree bound.

\section{Bounds for the averages and analytic continuation}\label{sec:continuation}
We now estimate the averages expressed by the identities in
Section~\ref{sec:D}. Summing first over $B$ and first over the rational
numerator gives two upper bounds. We combine them with the sum over
selected primes, keeping its signs, and treat principal characters
separately. The remaining subsections impose the local restrictions,
assemble the continued series, and apply Bochner's tube theorem.

\subsection{Uniform bounds for \texorpdfstring{$L$}{L}-functions}\label{app:bounds}
At primes outside the exceptional support, after removing the sign depending on the degree, the series in $B$ is
\[L_{q^2}(z,\chi_{A_*W/H_1})/L_q(z,\chi_{A_*W/H_1}|_{\F_q[T]}),
\]
which for a rational parameter reduces to $(1-qz)L_{q^2}(z,\chi)$. At an exceptional split prime, the normalized local factor is bounded on $|z|=r<1$ by a fixed power of
\[\frac{1+r^{\deg P/2}}{1-r^{\deg P/2}};
\]
the corresponding inert factors satisfy the analogous bound with $r^{\deg P}$. Their finite product is $q^{\eps d_{\rm bad}}$, where $d_{\rm bad}$ is the total degree of the exceptional support. This follows by separating finitely many primes of small degree and using geometric decay at the remaining primes. The same argument bounds the number of possible supports of the common divisor.

For a character of total modulus degree $D_0$ over $\F_Q$, the Euler product and the Riemann hypothesis for Dirichlet $L$-functions over function fields~\cite[Chapter 4 and Appendix]{Rosen} give logarithmic coefficients bounded by
\[\min(Q^r,(D_0+O(1))Q^{r/2}).
\]
On $|z|=Q^{-1/2-\delta_0}$, split the logarithmic series at $r\asymp2\log_Q(D_0+2)$. Both portions are $o(D_0)$, uniformly in the character. Thus $L(z)$ and its reciprocal, when the character is nonprincipal, are $O_{\delta_0,\eps}(Q^{\eps D_0})$ in the zero-free disk. For a principal character, the reciprocal is a polynomial, with any omitted Euler factors kept explicitly. This improves the elementary bound $(1+Q^{-\delta_0})^{D_0}$ and gives the Cauchy bounds used below, uniformly in the conductor.

\subsection{Two bounds for the averages}
We first estimate the sum over the rational numerator in the Fourier
expansion. The following proposition specifies the finite average and
all of its parameters.

\begin{proposition}\label{prop:second-trace}
Let $E\in\mathcal A_q$ be a nonconstant squarefree product of primes
that split over $\F_Q$, and write $\deg E=2n$.
For each $P\mid E$, choose one factor $\pi_P$ of
$P=\pi_P\pi_P^{\sigma_0}$ and an integer $j_P\in\{1,2,3\}$. Put
\[
 H_{\rm sel}=\prod_{P\mid E}\pi_P^{3-j_P}
                    (\pi_P^{\sigma_0})^{j_P-1},
 \qquad \theta=\prod_{P\mid E}\chi_{\pi_P}^{-j_P},
\]
where $\theta$ is restricted to $\F_q[T]$ and is zero on polynomials
not coprime to $E$. Let $H_R\in\mathcal A_q$ be fourth-power-free
with $(H_R,E)=1$, and set $h_R=\deg H_R$ and $h=h_R+2n$.
Fix integers $b,j\ge0$ and put $\delta=h+b-2j$.
For $F\in\mathcal A_{q^2,j}$ with $(F,E)=1$, define
\[
 C_b(F)=q^{-b}\sum_{M\in\mathcal A_{q,b}}
       \theta(M)G_{q^2}(H_RH_{\rm sel}M,F).
\]
Let $\mathcal S$ be a fixed finite set of primes in $\F_Q[T]$ to be
excluded from $F$. Then, for every $\eps>0$,
\begin{equation}\label{eq:twisted-degree-average-bound}
 \Biggl|
 \sum_{\substack{F\in\mathcal A_{q^2,j}\\(F,E)=1\\
                    \pi\nmid F\ (\pi\in\mathcal S)}} C_b(F)
 \Biggr|
 \ll_{q,\eps}
 q^{h/2+2j-\delta/2+\eps(h_R+n+j)}.
\end{equation}
The same bound holds after omitting the terms in the Fourier expansion
\eqref{eq:twisted-fourier-average} for which the character in the
variable $B$ is principal. More precisely, write
$F=AB$, where $A=(F,F^{\sigma_0})$, put $D_*=ABB^{\sigma_0}$,
$G=(H_R,D_*)$, $H_1=H_R/G$, and $W=V/G$.
Let $A_*=\prod_P P^{[v_P(A)]_4}$ be the fourth-power-free part of $A$,
where $[a]_4\in\{0,1,2,3\}$ is the least residue of $a$ modulo $4$.
The omitted terms are those for which
$\chi_{A_*W/H_1}$, regarded as a character of $B$ and extended by zero
at the excluded primes, is principal.
\end{proposition}
\begin{proof}
Use the decomposition $F=AB$ from the statement and write
$c=\deg A$ and $g=\deg G$. Let $\bar\theta=\theta^{-1}$ on units and
\[
 \mathfrak t=\sum_{a\bmod E}\bar\theta(a)e_q(a/E),
 \qquad |\mathfrak t|=q^n.
\]
Fourier inversion and $Y=H_{\rm sel}X\pmod F$ give the exact identity
\begin{equation}\label{eq:twisted-fourier-average}
 C_b(F)=\frac{\chi_F(H_{\rm sel})^{-1}\bar\theta(D_*)^{-1}}{\mathfrak t}
 \sum_{\deg V\le h+2n-\delta-c-1}
 \lambda(V)\bar\theta(V)\mathcal T_{F,H_R}(V/E),
\end{equation}
where $\mathcal T_{F,H_R}$ is defined in
\eqref{eq:trace-character-sum}, division by $E$ is taken modulo
$D_*$, and
\[
 \lambda(V)=\psi_q\!\left([T^{-b-1}]\frac{V}{ED_*}\right).
\]
Thus the monic leading-coefficient phase has not been discarded.
The local character sums in Section~\ref{sec:local-character-sums} imply $G\mid V$. Put $V=GW$ and
\[
 L=h+2n-\delta-c-g-1.
\]
If $L<0$ there is no contribution. Otherwise the weighted character sum equals
$q^{c+g}$ times a fixed unit phase, the primitive rational character
$\chi_B^{\rm prim}(W)$, and local functions $f_P(W)$.
At a split prime $P=\pi\pi^{\sigma_0}$ with unequal exponents in
$F$, choose $\pi$ so that
$v_\pi(F)=e+t$ and $v_{\pi^{\sigma_0}}(F)=e$, where $t>0$.
A nonzero local character sum requires $e=0$ or $4\mid e$, $e+t>v_P(H_R)$, and $P\nmid W$.
With $p=|P|$, the normalized functions are
\[
\begin{array}{c|c}
\text{condition at }P&f_P(W)\\\hline
\text{split},\ t>0,\ e=0&\mathbf1_{P\nmid W}\\
\text{split},\ t>0,\ e>0,\ 4\mid e&(1-p^{-1})\mathbf1_{P\nmid W}\\
\text{split with }t=0\text{, or inert},\ 4\nmid e&-p^{-1}+\mathbf1_{P\mid W}\\
\text{split},\ t=0,\ 4\mid e&(p-2)/p+p^{-1}\mathbf1_{P\mid W}\\
\text{inert},\ 4\mid e&1-p^{-1}\mathbf1_{P\mid W}.
\end{array}
\]
At split primes, $t=0$ means that the two exponents in $F$ are equal.
At inert primes, $e$ is the exponent of $P$ in $F$.
The inert degree signs belong to the fixed phase. All $|f_P|\le1$.
The phase can be retained without convention ambiguity as
\[
 C_F=q^{-c-g}\mathcal T_{F,H_R}(G/E)
       \left(\prod_{P\mid D_*}f_P(1)\right)^{-1};\qquad |C_F|=1.
\]
The character in $W$ is
\[
 \xi=\bar\theta\chi_B^{\rm prim},\qquad
 \operatorname{cond}(\xi)=E
 \prod_{\substack{P=\pi\pi^{\sigma_0}\text{ split}\\
 v_\pi(B)-v_{\pi^{\sigma_0}}(B)\not\equiv0\ (4)}}P.
\]
Its conductor degree is at most $2n+2(j-c)$. Since $(F,E)=1$, its nontrivial
ramification at $E$ cannot cancel; in particular $\xi$ is nonprincipal.

Write $f_P=a_P+b_P\mathbf1_{P\mid W}$. The exact monic generating function is
\[
 Z_F(u)=L_q(u,\xi)
 \prod_{\substack{P\mid D_*\\P\nmid\operatorname{cond}\xi}}
 (a_P+b_P\xi(P)u^{\deg P})
 \prod_{\substack{P\mid D_*\\P\mid\operatorname{cond}\xi}}a_P.
\]
An extra unit condition at an already weighted unramified prime replaces
$a_P+b_Pz$ by $a_P(1-z)$; it does not multiply the old factor by $1-z$.
At a previously unweighted unramified prime it inserts $1-z$.
On $|u|=q^{-1/2-\eps_0}$ the local products are subpower in their total
support degree. The function-field Riemann Hypothesis, applied to this
specified primitive rational character, gives logarithmic coefficients
bounded by
\[
 \frac{\min(q^r,(\deg\operatorname{cond}\xi+O(1))q^{r/2})}{r}.
\]
Let $z_v$ denote the coefficient of $u^v$ in $Z_F(u)$.
Splitting the logarithmic series at $r\asymp2\log_q(\deg\operatorname{cond}\xi+2)$
and applying Cauchy's formula gives
\[
 |z_v|\ll q^{(1/2+\eps_0)v+\eps(n+j)}.
\]
The exact sum with its monic phase is
\[
 (q-1)\sum_{v<L}z_v-z_L.
\]
Indeed the character and weights are invariant under $\F_q^\times$ scaling,
while the nontrivial top additive phase sums to $-1$.
For fixed $F$, the bound is therefore
$q^{-n+c+g+(L+1)/2+\eps(h_R+n+j)}$.

The inequality $c\ge g$ need not hold. Split $G=G_A G_B$, according as a shared
prime does or does not meet $A$. Then $g_A\le c$, whereas every
$P^d\mid G_B$ forces an oriented $\pi^{d+1}\mid B$.
For fixed orientations let $\ell_B$ be the mandatory extension degree in $B$;
then $2\ell_B\ge g_B+\deg\operatorname{rad}G_B$. The counts of $A$ and $B$
are at most $q^c$ and $q^{2(j-c-\ell_B)}$, respectively. The resulting exponent is
\[
 h/2+2j-\delta/2+(g-c)/2-2\ell_B
 \le h/2+2j-\delta/2.
\]
Divisor and orientation multiplicities are subpower in $h_R$.
Finally, the values of $W$ for which the character in $B$ is principal
have the form $W=aW_0K^4$ with fixed
fourth-power-free $W_0$. Their number is $O_q(q^{L/4})$ and $|f_P|\le1$,
so subtracting them does not exceed the proved square-root bound.
\end{proof}

For the contribution in which the character of $B$ is nonprincipal,
we obtain another bound by summing over $B$ first. Use the notation of
Proposition~\ref{prop:second-trace}, with $c=\deg A$ and $g=\deg G$. The twist has
conductor degree $2n$, and Fourier expansion of the character contributes $q^{-n}$. These primes enter the denominator in the trace condition, so the rational numerator has degree at most $h+2n-\delta-c-1$. Fixing the common divisor $G$ reduces the number of free numerator coefficients by $g$. For the estimate obtained by summing over $B$ first, we use the bounds
\[
 \begin{aligned}
  \#\{A\}&\le q^c,\\
  |\mathcal T_{F,H_R}(V/E)|&\le q^{c+g},\\
  \#\{W\}&\le q^{h+2n-g-\delta-c}.
 \end{aligned}
\]
together with $q^{(1+\eps)(j-c)+\eps(h+n)}$ for the $B$ coefficient. Multiplying and including $q^{-n}$ gives
\[q^{h+n+j-\delta+\eps(h+n+j)}
 =q^{h_R+3n+j-\delta+\eps(h_R+n+j)}.
\]

We now apply Proposition~\ref{prop:second-trace} together with the bound
obtained by summing over $B$ first. Fix the exceptional primes supporting the twist, with half their total degree equal to $n_D$. Sum over the remaining selected primes in~\eqref{eq:joint}, with half their total degree equal to $n_N$, retaining their signs. The joint Euler product gives one coefficient bound for this sum. The two resulting exponents are
\begin{align*}
 &h_R+4 n_N+3n_D+j-\delta,\\
 &h_R/2+3n_N+n_D+2j-\delta/2.
\end{align*}
The two exponents agree at $j=(h-\delta)/2+n_D$, where
$h=h_R+2(n_N+n_D)$. Including the weight in the degree sums gives
\[\beta h_R+[2+2\beta]n_N+
 \bigl(5-4\kappa+2\tau\bigr)n_D,
\]
and the remaining geometric sum in $\delta$ has ratio $q^{\kappa-3/2}$. This proves the estimate used in~\eqref{eq:Tplus}.

The principal character in the sum over $B$ is treated separately
in Section~\ref{sec:principal-characters}; its common prime divisors require the same distinction
between $G_A$ and $G_B$.

\subsection{Averaging over the selected primes}
We specify the joint sum used in the estimate obtained by summing over $B$ first. Fix
$X\in\F_q(T)^\times$ and a squarefree rational polynomial $E_0$
containing the zeros and poles of $X$ and every fixed excluded prime.
For $(B,E_0)=1$ define $\chi_X(B)=\prod_{\pi^e\parallel B}\chi_\pi(X)^e$,
and set it to zero otherwise. Use this same imprimitive convention
in all $L$-functions below. Its restriction to rational polynomials
coprime to $E_0$ is trivial. Put
$d_0=\deg\operatorname{num}X+\deg\operatorname{den}X+\deg E_0$,
with numerator and denominator coprime. Define
\[
 \mathfrak J_{X,E_0}(n,\ell)=
 \sum_{S,B}(-1)^{\omega(S)}\chi_X(S)^2\chi_X(B),
\]
where $S$ is squarefree over $\F_Q$, $\deg S=n$,
$(S,S^{\sigma_0})=(S,E_0)=1$, and $B$ is any monic polynomial
over $\F_Q$ with $\deg B=\ell$,
$(B,B^{\sigma_0})=(B,E_0)=(B,SS^{\sigma_0})=1$.
Thus a selected orientation excludes both conjugate primes from $B$,
but powers in $B$ are allowed. Let
\[
 \mathfrak J_{X,E_0}(\upsilon,z)
 =\sum_{n,\ell\ge0}\mathfrak J_{X,E_0}(n,\ell)\upsilon^nz^\ell.
\]
At a split prime $P=\pi\pi^{\sigma_0}\nmid E_0$, put
$\alpha=\chi_\pi(X)$, $\alpha'=\chi_{\pi^{\sigma_0}}(X)=\alpha^{-1}$,
$t=z^{\deg\pi}$ and $r=\upsilon^{\deg\pi}$.
The $B$ factor and the factor after summing the selected orientations are
\begin{equation}\label{eq:joint}
 F_B=\frac{1-\alpha\alpha't^2}{(1-\alpha t)(1-\alpha't)},\qquad
 F_B-(\alpha^2+(\alpha')^2)r.
\end{equation}
Define $U_P$ by dividing the second expression by
$F_B(1-\alpha^2r)(1-(\alpha')^2r)$. Then
$U_P=1+O(|rt|+|r|^2)$, and the exact global factorization is
\begin{align}\label{eq:joint-L-factors}
 \mathfrak J_{X,E_0}(\upsilon,z)
 &=\frac{L_{q^2}(z,\chi_X)}{L_q(z,\chi_X|_{\F_q[T]})
                 L_{q^2}(\upsilon,\chi_X^2)}\notag\\
 &\quad\times\prod_{\substack{P\text{ inert}\\P\nmid E_0}}
       (1-\chi_X(P)^2\upsilon^{\deg P})^{-1}
       \prod_{\substack{P\text{ split}\\P\nmid E_0}}U_P(\upsilon,z).
\end{align}
This is initially an identity in the absolute-convergence polydisk.
The last two products converge normally for
$Q|\upsilon z|<1$, $Q|\upsilon|^2<1$, and $q|\upsilon|<1$.
On $|\upsilon|=|z|=q^{-1-\eps_0}$, the uniform $L$ bounds and
two-variable Cauchy formula give, when $\chi_X$ is nonprincipal,
\begin{equation}\label{eq:joint-bound}
 |\mathfrak J_{X,E_0}(n,\ell)|
 \ll q^{(1+\eps_0)(n+\ell)+\eps d_0}.
\end{equation}
If $\chi_X^2$ is principal, its reciprocal $L$-function is the
polynomial $1-Q\upsilon$ with the explicit omitted Euler factors,
and the same bound holds. If $\chi_X$ itself is principal, its
$B$ counting pole is handled in the next subsection. Fixed
exclusions are already absent in \eqref{eq:joint-L-factors}; no
division by a possibly vanishing joint local factor is used.

We now identify this sum with the selected primes in
\eqref{eq:hecke}. The critical rational numerator valuation is
$a=1$, and its Fourier exponent is $j_P=2$.
For two selected orientations $\pi,\rho$, put
$x_0=\chi_\pi(\rho)$ and $y_0=\chi_\pi(\rho^{\sigma_0})$.
For general exponents $j=a+1$, $j'=a'+1$, the Hecke cross phase
and the Fourier cross phase are both
$x_0^{j+j'}y_0^{j-j'}$, so they cancel. The remaining local
phase is $\kappa_q^{-j^2\deg\pi}=1$ for $j=2$.
Before extracting a common divisor, the actual parameter is
$X=AV/H_R$. Outside the fixed exceptional support, the $B$ factor
is $\chi_X(B)$ and the selected factor at $\pi$ is
$-\chi_\pi(X)^2$. The explicit local table factors on the
exceptional support remain separate. In particular a fourth-power
factor of $A$ does not impose a new zero condition on $B$.
The factor $H_R$ is retained. The factors from the selected
part of the numerator and Fourier inversion leave exponents
$(-4,0)$ at the conjugate pair, and hence disappear by quartic
periodicity on units. After extracting the common divisor $G$, use the reduced
parameter $X=A_*W/H_1$, with the extracted norm powers and fixed
exceptional support unchanged.

For fixed degrees $n_N,n_D,c,j,\delta$, the degree bound on $V$ is
\[
 \deg V\le h+2(n_N+n_D)-\delta-c-1,
 \qquad h=h_R+2(n_N+n_D),
\]
decreased by the fixed gcd degree after extraction. This depends
only on the degree data, not on the positions or orientations of
the remaining selected primes. One may therefore fix $A,V,H_R$
and the exponents at the exceptional primes first, apply \eqref{eq:joint-bound},
and then sum over the remaining degree parameters. Monicity makes the
remaining additive phase depend only on the leading coefficient
and the degree bound, also independently of those prime positions.
Selected primes dividing $V$ have zero Fourier character; primes
dividing $A$ or $H_R$ are included in $E_0$.

Finally the unit-inclusion scalar in \eqref{eq:unit-deviation} is
not discarded before this signed average. For its divisor $J$,
$\theta(J)=\theta_{\rm fix}(J)\prod_{\pi\mid S}\chi_\pi(J)^{-2}$,
where the fixed-support scalar $\theta_{\rm fix}(J)$ is retained
outside the signed sum. The variable part replaces the selected
parameter by $X/J$. At the same time $H_R$ is replaced
by $H_RJ$, so the $B$ parameter also becomes $X/J$. Thus the
two characters still have exactly the common parameter required
by \eqref{eq:joint-L-factors}. All removed terms at the exceptional primes, exclusions
and degree shifts are fixed before applying this argument.

\subsection{The contribution of principal characters}\label{sec:principal-characters}
Let $H_R$ be the reduced numerator in $\F_q[T]$, write $h_R=\deg H_R$, and suppose that the primes at which one of the two conjugate factors is selected have total degree $2n$ over $\F_q$. Put $h=h_R+2n$, $\delta=h+\deg M-2j$. Before summing over these primes, the two estimates for the character sums are
\begin{gather*}q^{h_R+3n+j-\delta+\eps(h_R+n+j)},\\
 q^{h_R/2+n+2j-\delta/2+\eps(h_R+n+j)}.
\end{gather*}
We use the smaller of these two bounds and multiply it by $q^{-\tau h-2\sigma j+\tau\delta}$. After summing over the selected primes with their signs, we split the sum at the value of $j$ where the two bounds agree. The resulting factors are $q^{\beta h_R}$ and $q^{[2+2\beta]n}$, where $\beta=3/2-\kappa$.

Put $G=(H_R,D_*)=G_A G_B$ as above, with
$g_A=\deg G_A$, $g_B=\deg G_B$ and
$f_A=\deg\operatorname{rad}G_A$, $f_B=\deg\operatorname{rad}G_B$.
If $\ell_B$ is the degree of the compulsory oriented divisor of $B$,
the local restrictions give
\[
 c\ge g_A+f_A,\qquad 2\ell_B\ge g_B+f_B.
\]
After restoring the fourth-power factors of $A$, principality is
$AW=c_0(H_R/G)L^4$. For each $L$, $A$ is a divisor of
$(H_R/G)L^4$, with only a subpower number of choices. The trace
cutoff gives $\deg L\le n-\delta/4$, and
$\operatorname{rad}G_A\mid L$, so $\delta\le4(n-f_A)$.
There is no such radical restriction at $G_B$.

The principal $B$ series is the counting function
$(1-qz)/(1-Qz)$ with its specified finite local factors.
Its compulsory monomial of degree $\ell_B$ contributes
$q^{-2\ell_B}$ at $z=Q^{-1}$. Subtracting that pole leaves a
function estimated on $|z|=q^{-1-\eps_0}$ by the preceding
coefficient argument. After exact geometric continuation and
multiplication by $C(s)$, the counting-pole exponent is
\begin{align*}
 E={}&-\tau h+(2-2\sigma)\max\{c+\ell_B,(h-\delta)/2\}\\
 &-c+g_A+g_B-f_A-2\ell_B+(\tau-1/4)\delta.
\end{align*}
In particular this is not a bound for an increasing positive
geometric tail in its region of divergence.
If the maximum is $(h-\delta)/2$, then
\[
 E\le(1-\kappa)h_R+(2\kappa-3)n\le0.
\]
If the maximum is $c+\ell_B$, then
\begin{align*}
 E\le{}&-\tau h_R+(2\tau-1)n+(2-2\sigma)g_A+(1-\sigma)g_B\\
 &-(2\sigma+4\tau-1)f_A-\sigma f_B.
\end{align*}
Using $g_A+g_B\le h_R$ and $2-2\sigma\le\tau+\beta$ proves
\begin{equation}\label{eq:principal}
 q^{\beta h_R+\max(0,(2\tau-1)n)+\eps(h_R+n)}
\end{equation}
as the bound for this contribution. Including the original weight
and the number of selected primes gives $2-4\sigma<0$ when
$\tau>1/2$, and $3-2\kappa-2\sigma<0$ otherwise.

\subsection{Local restrictions on the modulus}
For $b\ge0$ and a fixed modulus degree $j$, define
\begin{equation}\label{eq:complete-averages}
\begin{aligned}
 C_b(H;j)&=q^{-b}\sum_{\deg M=b}C_4(HM,j),\\
 c_H(j)&=C_b(H;j)\quad(b\ge2j),\qquad
 d_{H,b}(j)=C_b(H;j)-c_H(j).
\end{aligned}
\end{equation}
The additive character orthogonality in Section~\ref{app:fibres} proves
that $c_H(j)$ is independent of $b\ge2j$. If a prime is excluded
from the modulus $F$, the same exclusion is retained in every
term of these definitions. For a remaining character $\theta$
of conductor $E$, the twisted average is defined by inserting
$\theta(M)$ in the sum; its complete average is taken after the
summation covers a complete system of residues. This convention
also applies to the local identities below.

The degree range in the kernel continues to use the original numerator
degree until all local transformations and exclusions have been recorded. At a split
prime $P=\pi\pi^{\sigma_0}$, write $n=\deg\pi$. The nontrivial
terms have rational parameter $P^a$ ($0\le a\le2$), or parameter
$\pi^{3-j_P}(\pi^{\sigma_0})^{j_P-1}$ with twist
$\chi_\pi^{-j_P}$ ($1\le j_P\le3$).
For a single local modulus factor of degree $en$, the transformation is
\[
 h'=h-2en+4n,\qquad j'=j-en,\qquad\delta'=\delta+4n.
\]
The term selecting both conjugate primes has $\delta'=\delta+8n$ instead. In particular,
for parameter $P$ and local factor $P^2$, one has $h'=h$,
$j'=j-4n$, and the quadratic twist is retained; this is not an
untwisted term.

Put $\Gamma=5-4\kappa+2\tau$. Relative to rational cost
$2a\beta n$ and twisted cost $\Gamma n$, the local Gauss value
and the weight $q^{-2\kappa\deg L}$ give the following exponents
per $n$:
\[
\begin{array}{c|c}
\text{transition}&\text{exponent}\\\hline
P^a\text{ to a twisted parameter}&6-a-4\kappa-2\sigma\\
P^a\text{ to }P^{2-a}\text{ (double)}&8-2a-8\kappa\\
\text{twisted to }P^a\text{ (single)}&a-4\kappa+2\sigma\\
\text{twisted to the opposite orientation (double)}&6-8\kappa
\end{array}
\]
All four are strictly negative in \eqref{eq:Tplus}. The case of a rational numerator with valuation three is treated separately below. For a critical
prime from $N$ with $a=1$, a term selecting one conjugate prime and giving a rational parameter has total
exponent $7-8\kappa-2\sigma<-4$ before the prime count. A term selecting both conjugate primes, counted absolutely with its remaining twist and prime count,
has exponent $14-12\kappa-4\sigma<-3$.
Thus only the remaining leading critical primes need the signed
sum \eqref{eq:joint}. The finitely many local terms are fixed
before that sum; no unrestricted functional equation is applied
to a restricted Gauss series.

Fix the local factor $L$ of the modulus and write $F=LF_0$ with $(F_0,P)=1$. Suppose that $M$ is a unit at $P$. The Chinese Remainder Theorem gives
\begin{equation}\label{eq:lowCRT}
 G_{q^2}(HM,LF_0)=G_{q^2}(H,L)\overline{\chi_L(M)}\,
 G_{q^2}(HL^{2}M,F_0).
\end{equation}
Since $F_0$ is coprime to $P$, we may reduce the valuation of $HL^{2}$ at $P$ modulo $4$. Apply \eqref{eq:lowCRT} to both $C_b(H;j)$ and $c_H(j)$. Their difference $C_b(H;j)-c_H(j)$ satisfies the same identity, with the same scalar and character, for every $b$. Thus we may factor the modulus before subtracting the complete average.

For the untwisted averages over $\F_q[T]$, the condition that $M$ is a unit gives
\begin{equation}\label{eq:unitaverage}
 C_b(H;j;P\nmid M)=C_b(H;j)-q^{-2n}C_{b-2n}(HP;j),\qquad \deg P=2n,
\end{equation}
where the left side restricts the sum over $M$ to $P\nmid M$, with the modulus $F$ still coprime to $P$ and the second term zero for $b<2n$. If the conductor of $\theta$ is coprime to $P$, the second term is multiplied by $\theta(P)$, which has modulus one. Combining its generating weight with the estimate for the numerator gives
$q^{(2\tau-2+2\beta)n}=q^{(1-2\sigma)n}<1$.
For a squarefree unit support $U$ coprime to the conductor of $\theta$,
write $d_b^U(H,\theta;j)$ for the restricted degree average minus
its eventual complete average. Exact inclusion-exclusion gives
\begin{align}\label{eq:unit-deviation}
 d_b^U(H,\theta;j)
 &=\sum_{\substack{J\mid U\\\deg J\le b}}
 \mu(J)q^{-\deg J}\theta(J)d_{b-\deg J}(HJ,\theta;j)\notag\\
 &\quad-\sum_{\substack{J\mid U\\\deg J>b}}
 \mu(J)q^{-\deg J}\theta(J)c(HJ,\theta;j).
\end{align}
Every inherited exclusion on the modulus $F$ remains in each term.
The second sum is required because the complete average includes
all divisors $J$, whereas a monic polynomial of degree $b$ cannot
be divisible by a larger $J$. The phase $\theta(J)$ is retained
before applying \eqref{eq:joint} to any remaining selected primes.
The first sum costs $q^{(1/2-\sigma)\deg J}$ relative to the rational
parameter bound. For the finite second sum, the complete-average
estimate
\[
 c(HJ;j)\ll q^{(2+\eps)j+\eps(h+\deg J)}
\]
and the original degree weight give, after summing $b<\deg J$ and $j$,
\[
 q^{-(1-\max(\tau,0))\deg J+\eps(h+\deg J)}.
\]
Both margins are strict on \eqref{eq:Tplus} and \eqref{eq:bridge}.

If no selected support remains, use the untwisted estimate directly.
The deviation is zero when $\delta\ge h$, and otherwise
\[
 |d_{H,b}(j)|\ll q^{h-\delta+j+\eps(h+j)}.
\]
Its weighted exponent is
$(1-\tau)(h-\delta)+(1-2\sigma)j$. Since $\sigma>1/2$,
summing the decreasing $j$ tail from $j\ge(h-\delta)/2$ gives
$q^{\beta h+\eps h}$, or $q^{\beta_+h+\eps h}$ on the bridge.
No nonprincipal character at a selected prime is invoked in this case.

A unit inclusion can raise a rational valuation from three to four.
Write $H=H_0K^4$, $(K,F)=1$, $v=\deg K$, and keep
$h=h_0+4v$ in the original kernel and cutoff. The shared gcd comes
only from $H_0$. The nonprincipal estimate is unchanged; for the
principal term the root cutoff is $\deg L\le v-\delta/4$ and
$\delta\le4(v-f_A)$. Its exponent is
\begin{align*}
 -\tau h+(2-2\sigma)\max\{c+\ell_B,(h-\delta)/2\}
 -c+g-f_A-2\ell_B+v+(\tau-1/4)\delta.
\end{align*}
The two branches are bounded by $(1-\kappa)h_0\le0$ and
$\beta h_0\le\beta h$, respectively, by the corrected gcd
inequalities. The original degree is not replaced by $h_0$.
The Ramanujan sum for numerator valuation three has local modulus exponent four
and Gauss norm $q^{6n}$. Its single and double exponents are
$6-8\kappa$ and $12-16\kappa$, with respective shifts
$\delta+8n$ and $\delta+16n$. All fourth-power exclusions are retained.
On the bridge use $\beta_+$; when $\sigma\ge1$, the coefficient
$2-2\sigma$ is nonpositive and the lower cutoff gives the bound.

For several selected primes, fix these local factors first, group
the remaining primes by total degree, and apply \eqref{eq:joint}.
The degree shifts and the phase in \eqref{eq:unit-deviation} are
kept in this finite decomposition.

\subsection{Local Euler factors}\label{app:stable}
For a prime $P$ put $u=(\kappa_q z)^{\deg P}$, $v=y^{\deg P}$,
\[r=puv^2=p^{1-2s},\qquad t=v/p.
\]
Let $F_a$ denote the factor for the complete average, normalized over units, with the prime excluded and reduced rational numerator valuation $a$. For a split prime set $B=p^{3}u^2$, $A=Bv^2$. Then
\begin{align*}
 t^aF_a&=\frac{p^{-a}(1-A)^2+r^{a+1}(1-v^{-2})^2}{(1-B)^2},\quad 0\le a\le2,\\
 t^{3}F_{3}&=\frac{p^{-3}-2p^{-2}r^2+p^{-1}r^4}{(1-B)^2},\\
 puF_{2}&=\frac{rv^{-4}(1-A)^2+A^2(1-v^{-2})^2}{(1-B)^2}.
\end{align*}
For an inert prime set $B=p^{6}u^4$, $C=Bv^4$; the corresponding formulas are
\begin{align*}
 t^aF_a&=\frac{p^{-a}(1-C)+r^{a+1}(1-v^{-4})}{1-B},\\
 t^{3}F_{3}&=\frac{p^{-3}-p^{-1}r^4}{1-B},\\
 puF_{2}&=\frac{rv^{-4}(1-C)+C(1-v^{-4})}{1-B}.
\end{align*}
Including the density gives the local factor
\[(1-p^{-1})\left[\frac{\sum_{a=0}^{3}t^aF_a}{1-t^4}
 +(1+r)-F_0-puF_{2}\right].
\]
These identities follow by inserting the prime power Gauss values in~\eqref{eq:hecke}; their residue specializations are~\eqref{eq:JP}.

Let $G(r)=\sum_{a=0}^{3}r^a$. The terms independent of $r$ cancel at order $p^{-1}$ because
\[(1-p^{-1})\left[G(r)+\sum_{a=1}^{3}p^{-a}\right]
 =G(r)-p^{-1}(G(r)-1)-p^{-4}.
\]
Division by $G(r)$ is uniform on compact subsets of $\sigma>1/2$. Expanding the remaining terms gives the nine errors displayed below. Their strict summability conditions are
\begin{gather*}
 2\sigma>1,\quad4(1-\tau)>1,\quad\kappa>1,\quad
 \sigma+\kappa>3/2,\\
 4\sigma+2\tau>3,\quad4\sigma>1,\quad
 2\sigma+4\tau>2,\quad\sigma+\kappa>7/4.
\end{gather*}
They follow from the region \eqref{eq:Tplus}. The product of $G(r)$ is~\eqref{eq:stableleading}. On the connecting region \eqref{eq:bridge}, where $\tau$ may be negative, retain both $v^{-2}$ and $v^{-4}$. Every displayed convergence inequality remains strict there. The stronger inequalities $4\sigma+4\tau>3$ and $2\sigma+4\tau>2$ also hold, so the same local factors converge normally on that region.

\subsection{Euler products for complete averages}
In the term with $P\nmid N$, the modulus $F$ may still be divisible by $P$. Including this term, the leading factor of the complete average at each prime $P$, whether split or inert over $\F_Q$, is
\[1+r_P+\cdots+r_P^{3},\qquad r_P=p^{1-2s}.
\]
Its Euler product is
\begin{equation}\label{eq:stableleading}
 \frac{1-q^{5-8s}}{1-q^{2-2s}}.
\end{equation}
After extracting this rational factor, the remaining Euler product is normally convergent in the tube below. Expanding the local factors at the split primes bounds the error by
\[p^{-1}|r_P|+p^{-4}+|t_P|^4+|B_P|+p^{-1}|A_P|
 +|r_P|^2|v_P|^{-2}+p^{-2}|r_P|^{2}
 +|r_P||v_P|^{-4}+|A_P|^2,
\]
where $v_P=p^w$, $t_P=v_P/p$, $B_P=p^{3-4(s+w)}$, $A_P=B_Pv_P^{2}$. These terms are summable in the region below; in particular, $\sigma+\kappa>7/4$. The errors at inert primes are smaller. Summing the higher powers in $N$ requires $\tau<3/4$. The finite terms introduced by extending the degree sums satisfy $\delta<h-2j$. Evaluating the complete average at $q^{-2\kappa}$ bounds them under weaker convergence conditions.

\subsection{The decomposition of each transformed series}
We make explicit how the complete averages and their corrections
represent the function in \eqref{eq:dual}. First apply the finite
identity \eqref{eq:hecke}, retaining its Gauss scalars, exclusion
denominators and the original weights in $N,D$. The order of the summations is as follows. Fix the original pair
$(N,D)$ in \eqref{eq:dual}; this fixes $S$, the reduced valuations
$a_\pi$, and the original degree weights. Next sum over
$I\subseteq S_0$ in \eqref{eq:hecke}, with its displayed sign,
Gauss scalar and denominator. At each selected prime, partition
the sum over moduli $F$ according to the exponent of that prime in $F$ and use
\eqref{eq:lowCRT} on that part; these parts are disjoint.
Finally sum over the divisors $J$ of the remaining unit support
as in \eqref{eq:unit-deviation}, retaining $\mu(J)\theta(J)$
and its degree shift. Thus the subset and divisor sums are finite
identity expansions, not additional copies of an original term.
The fourth-power part removed from a numerator is still recorded
in its original degree weight and exclusion set. Only after these
indices have been fixed do we group the remaining primes by degree
and use \eqref{eq:joint}.

After this inclusion-exclusion, each term has an unrestricted
series with fixed monic parameter $H$, $h=\deg H$, and a character
$\theta$ in the rational variable. Put $y=qx$, $R=\kappa_qQz$ and
$\lambda_0=\kappa_qQ^{-1}R$. Define
\[
 C_b^\theta(H;j)=q^{-b}\sum_{\deg M=b}\theta(M)C_4(HM,j),
 \qquad d_b^\theta(H;j)=C_b^\theta(H;j)-c^\theta(H;j),
\]
where $c^\theta(H;j)$ is the eventual complete average.
If $E$ is the conductor of $\theta$, orthogonality of additive characters gives
$d_b^\theta(H;j)=0$ for $b\ge2j+\deg E$.
Substitution of \eqref{eq:halflength}, followed by $a=h+b=2j+\delta$,
gives the exact identity
\begin{align}\label{eq:full-decomposition}
 &\sum_{M\in\mathcal A_q}\theta(M)x^{\deg M}
       \Psi_{q^2}(HM,\kappa_qR/Q)\notag\\
 &=y^{-h}\mathcal K_4(y,R)
       \sum_{j\ge0}c^\theta(H;j)(\lambda_0y^2)^j\notag\\
 &\quad+\sum_{b\ge0}y^b\sum_{0\le j\le(h+b)/2}
       \lambda_0^jK_{4,h+b-2j}(R)d_b^\theta(H;j)\notag\\
 &\quad-y^{-h}\sum_{0\le a<h}y^a\sum_{0\le j\le a/2}
       \lambda_0^jK_{4,a-2j}(R)c^\theta(H;j).
\end{align}
The last line subtracts exactly the artificially inserted degrees
$a<h$; the negative powers in the separated terms cancel in the
identity. The correction in the third line has
$\delta<h+\deg E$. Equivalently, if unit restrictions are retained during complete-average
subtraction, \eqref{eq:unit-deviation} accounts for their finite
degree ranges. These are the same terms already present after prior
inclusion-exclusion, not a second correction.

These are first identities of convergent series for sufficiently
small $x,z$, equivalently of their power-series coefficients.
The complete-average Euler products above, the short-degree bounds,
and the finite compensation estimates continue this same expression.
Local restrictions in a low modulus decompose its coefficients
using \eqref{eq:lowCRT}; they are not an application of the
unrestricted functional equation to a restricted series.
The apparent factor $(1-y)^{-1}$ in the complete term is canceled by
the original prefactor, since
\[
 \frac{\mathcal P(w)}{1-y}
 =-\frac{q^{w-1}}{1-q^{w-1}},\qquad y=q^w.
\]
The terms over finite degree ranges have no such denominator. Thus this factor
introduces no pole on the tubes used here, where $\tau<3/4$.
We retain the scalar factors in \eqref{eq:hecke}, including their phases,
and the original degree weights before grouping the terms. Thus the complete and correction terms are
attached to the original dual function, rather than defined as
independent continuations.

\begin{lemma}[Analytic continuation]\label{lem:tube}
The product $C(s)G_4(s,w)N_4(s,w)\widehat A_4(s,w)$ is holomorphic on the connected union of the region obtained by coefficient interpolation, the connecting region \eqref{eq:bridge}, and the tube over the following open set $R_1$:
\begin{equation}\label{eq:Tplus}
 \frac12<\sigma<1,\quad \frac54<\kappa<\frac32,
 \quad4\kappa+2\sigma>6,\quad\tau<\frac34.
\end{equation}
Normal convergence is uniform on compact subsets in all imaginary directions.
\end{lemma}
\begin{proof}
The estimate for the sum over selected primes gives the third inequality in~\eqref{eq:Tplus}. Absolute convergence of the sum over $D$, with $H_R=D^{2}$, gives the second. We use~\eqref{eq:principal} for principal characters. The other valuations, the terms selecting both conjugate primes, the inert primes, and the finitely many modified local factors have the smaller exponents established above. Formula~\eqref{eq:stableleading} handles the complete averages, and the remaining finite sums satisfy the stated bounds. Multiplying by the factors in \eqref{eq:clearing} therefore gives a normally convergent expression.

It remains to connect this region with the domain of
\eqref{eq:dual}. Consider the open set
\begin{equation}\label{eq:bridge}
 \frac{99}{100}<\sigma<\frac{33}{20},\qquad
 \frac{149}{100}<\kappa<\frac{151}{100}.
\end{equation}
Here $\tau=\kappa-\sigma<13/25$. Put
$\beta=3/2-\kappa$ and $\beta_+=\max(\beta,0)$.
When the selected primes have total degree $2n$, the estimate obtained by
summing first over the rational numerator $V$, including the
weight, has exponent
\[
 2n+(1/2-\tau)h+(2-2\sigma)j+(\tau-1/2)\delta,
 \qquad h=h_R+2n.
\]
The correction involves the finite interval $0\le\delta<h+2n$.
Put $d=h-\delta$, $j_0=\max(0,d/2)$ and $j_*=d/2+n$.
Then $j_*>0$ and $0\le j_*-j_0\le n$. Including the absolute
support count, the two estimates for the character sums have exponents
\[
 E_B(j)=(1-\tau)d+3n+(1-2\sigma)j,\qquad
 E_V(j)=(1/2-\tau)d+2n+(2-2\sigma)j.
\]
Their difference is $E_B(j)-E_V(j)=j_*-j$. We use $E_V$ for
$j_0\le j\le j_*$ and $E_B$ for $j\ge j_*$. The first range
contains $O(n+1)$ integers, including when $\sigma=1$; its maximum
occurs at an endpoint. The second range is a geometric tail with
ratio $q^{1-2\sigma}<q^{-49/50}$ throughout \eqref{eq:bridge}.
The factors $q^{\eps(h_R+n+j)}$ in the bounds for the character sums are
retained by choosing $\eps$ sufficiently small, so the tail
ratio remains bounded away from one.
Integer endpoints change the bound by a constant uniform on compact
subsets. At the upper endpoint the exponent is
\[
 E_B(j_*)=E_V(j_*)=\beta d+(4-2\sigma)n.
\]
At the lower endpoint it is $\beta d+2n$ if $d\ge0$, and
$(1/2-\tau)d+2n$ if $d<0$. Since
$-2n<d\le h_R+2n$, each of these expressions is at most
$\beta_+h_R+(51/25)n$ on \eqref{eq:bridge}. Thus the same bound
holds uniformly across $\sigma=1$, up to polynomial factors in
$h_R+n+1$. These factors, including the length of the $\delta$
sum, are absorbed by arbitrarily small positive exponents in the
strict convergence estimates below. Restoring the original support weight
$q^{(2\tau+1-4\kappa)n}$ therefore gives exponent
\[
 \frac{76}{25}-2\sigma-2\kappa<-\frac{48}{25}
\]
in $n$.
For the outer sum over $D$, we have $h_R=2\deg D$; counting
$D$ and restoring its weight gives
\[
 \bigl(2-2\kappa+2\beta_+\bigr)\deg D
 <-\frac{24}{25}\deg D.
\]
The complete average is summed using \eqref{eq:kernel}, and the
local Euler factors satisfy the strict convergence inequalities
in Section~\ref{app:stable} throughout \eqref{eq:bridge}.
These estimates are uniform in the imaginary parts and prove normal convergence on the corresponding tube.
The point $(\sigma,\kappa)=(199/200,299/200)$ lies in its
intersection with \eqref{eq:Tplus}. The point
$(\sigma,\kappa)=(8/5,301/200)$ has $\tau<0$ and lies in
the original domain of \eqref{eq:dual}. The point
$(\sigma,\kappa)=(6/5,299/200)$ lies in the region obtained
by coefficient interpolation below. All three intersections are open. The connecting region also meets
$S_2$ at $(\sigma,\tau)=(201/200,1/2)$. The expressions agree
there by continuation from the original dual domain, so $S_2$
may be included in the connected union. Thus the pole-free expressions agree on their overlaps and
define one holomorphic function on the connected union.
\end{proof}
\subsection{Coefficient interpolation}\label{app:oldtube}
We now derive the second region used in the convex hull. Expand the coprime Gauss sum with squarefree modulus $F=E B$, where $E$ is rational and $B$ has no rational divisor. Put $\eta=\overline{\chi_B}|_{\F_q[T]}$ and $\xi=\eta^{2}$. The normalized difference in the rational numerator variable has divisor average
\begin{align*}
 S_F(u,v)&=R_F(u)\prod_{P\nmid E}
 [1-\xi(P)v^{\deg P}+\eta(P)\xi(P)(uv)^{\deg P}],\\
 R_F(u)&=(1-qu)L_q(u,\eta)\prod_{P\mid E}(1-\eta(P)u^{\deg P}).
\end{align*}
When $\eta$ is principal, interpret $R_F$ as the finite product after canceling $1-qu$. Factoring the product gives $L_q(v,\xi)^{-1}$, finite omitted factors, and
\[\prod_{P\nmid E}\left(1+
 \frac{\eta(P)\xi(P)(uv)^{\deg P}}{1-\xi(P)v^{\deg P}}\right).
\]
The last product converges normally for $|v|<1$, $q|uv|<1$. On circles of radius $q^{-1/2-\delta}$ in both variables,
Section~\ref{app:bounds} and the finite Euler factors first give
$q^{(1/2+\delta)(a+b)+\eps\deg F}$. For a coefficient of
$v$-degree $b$, the $u$-degree is at most
$\deg R_F+b\le2\deg F+1+b$, since each selected $u$ power in
the Euler factors has the same $v$ degree. Thus the extra
$\delta a$ is absorbed into $2\delta\deg F+\delta b+O(\delta)$.
After relabeling the small constants, the coefficient of $u^av^b$
in $S_F(u,v)$ is
\[
 O\!\left(q^{a/2+(1/2+\eps)b+\eps\deg F}\right).
\]
When $B\ne1$, $\xi$ has local order two; when $B=1$, the reciprocal of its $L$-function is the polynomial $1-qv$. Expanding the Gauss series at $\kappa=3/2+\eps$ and including the factor $q^b$ gives the right endpoint bound $q^{a/2+(3/2+\eps)b}$.

At the left endpoint, the individual Gauss sum estimate after removing the pole is
\[|(1-q^{10-8k})\widetilde\Psi_{q^2}(ND^{2},\kappa_qQ^{-k})|
 \ll |N|^{3/2-\Re k+\eps}|D|^{2(3/2-\Re k)+\eps}.
\]
For $\Psi_{q^2}$, this follows from the finite expression: the prime power formulas give $|C_4(H,j)|\ll q^{3j+\eps\deg H}$ after summing over divisors of $H$, and each term obtained by the functional equation is bounded by $q^{(3/2-\Re k)\deg H+\eps\deg H}$. Apply~\eqref{eq:hecke} to impose coprimality. At a reduced nonzero valuation $a$, the ratio of the bound for a transformed numerator to the original bound is a power of $q^{\deg\pi}$ with exponent
\[4-4\Re k-a.
\]
It is nonpositive for all $a\ge1$ when $\Re k>3/4$. If a positive valuation reduces to zero modulo four, the weight from the original exponent absorbs this ratio. The Hecke denominators are controlled by divisor bounds for $\Re k>3/4$. This proves the left endpoint bound at $\kappa_0=3/4+\eps$.

Counting $N,D$ and including $q^b$, the left endpoint coefficient bound is
\[q^{(5/2-\kappa_0)a+[2+2(3/2-\kappa_0)]b+\eps(a+b)}.
\]
Interpolation by the three circles theorem in the Gauss variable, after removing the pole, between $\kappa_0$ and $3/2$ therefore gives
\[p_4=\frac{2-\kappa_0}{3/2-\kappa_0},\qquad
 r_4=2+\frac{1}{2(3/2-\kappa_0)}.
\]
Letting the endpoint displacement tend to zero yields $p_4=5/3$ and $r_4=8/3$. Multiplication by $x^a q^{-2\kappa b}$ gives
\[5\sigma+2\tau>6,\qquad\kappa>\frac{33}{28}.
\]
The finite difference contributes the factor $1-qx$. Its possible
zeros introduce no additional poles in the moment, since
\[
 \frac{\mathcal P(w)}{1-qx}
 =-\frac{q^{w-1}}{1-q^{w-1}},\qquad x=q^{w-1},
\]
which is holomorphic throughout this region, where $\tau<1/2$.
The two boundary lines intersect at $(17/14,-1/28)$, the point used in the convex hull. All estimates hold on full circles, so the interpolation and the tube argument are uniform in the imaginary parts.

\subsection{Convergence region and convex hull}
The coefficient interpolation in Section~\ref{app:oldtube}
gives the following open set $R_2$:
\[
 5\sigma+2\tau>6,\qquad \sigma+\tau>\frac{33}{28},
 \qquad \frac34<\sigma+\tau<\frac32,
\]
with boundary point $(17/14,-1/28)$. The closure of the region
\eqref{eq:Tplus} contains $(1/2,3/4)$. This point is approached from within~\eqref{eq:Tplus}, for example by $\sigma=1/2+2\eps$, $\kappa=5/4+\eps$.

Theorem~\ref{Bochner} applies to the connected domain of the pole-free function. Figure~\ref{fig:hull} shows the two regions and the relevant segment. At $\tau=1/2$, the segment between these two points gives
\[
 \frac{15}{22}\left(\frac12,\frac34\right)
 +\frac7{22}\left(\frac{17}{14},-\frac1{28}\right)
 =\left(\frac8{11},\frac12\right).
\]
The estimates are uniform in the imaginary parts on compact subsets
of the real base. The tube theorem therefore gives continuation to
the tube over the convex hull. For every $\sigma>8/11$ on the
central line, interior points sufficiently close to the two limiting
points, together with $S_2$ to the right, give an open neighborhood
in this hull. No holomorphy at the boundary point $(8/11,1/2)$
is asserted.

\begin{figure}[ht]
\centering
\begin{tikzpicture}[x=8cm,y=5cm,>=stealth,font=\small]
  \fill[blue!9] (.5,.75)--(.75,.75)--(1,.5)--(1,.25)--cycle;
  \fill[orange!14] (1,.5)--(1.214285714,-.035714286)--(1.45,-.271428571)--(1.45,.05)--cycle;
  \draw[->] (.4,-.31)--(1.51,-.31) node[right] {$\Re s$};
  \draw[->] (.4,-.31)--(.4,.9) node[above] {$\Re w$};
  \draw[blue!65!black,dashed] (.5,.75)--(.75,.75)--(1,.5)--(1,.25)--cycle;
  \draw[orange!70!black,dashed] (1,.5)--(1.214285714,-.035714286)--(1.45,-.271428571);
  \draw[orange!70!black,dashed] (1,.5)--(1.45,.05);
  \draw[gray,densely dashed] (.4,.5)--(1.43,.5);
  \node[left] at (.4,.5) {$\frac12$};
  \node[blue!65!black] at (.81,.61) {$R_1$};
  \node[orange!70!black] at (1.31,.035) {$R_2$};
  \draw[very thick] (.5,.75)--(1.214285714,-.035714286);
  \fill (.5,.75) circle (1.7pt);
  \node[above right] at (.5,.77) {$\left(\frac12,\frac34\right)$};
  \fill (1.214285714,-.035714286) circle (1.7pt);
  \node[below left] at (1.214285714,-.06) {$\left(\frac{17}{14},-\frac1{28}\right)$};
  \fill[teal!80!black] (.727272727,.5) circle (2pt);
  \draw[teal!80!black] (.727272727,.5)--(.62,.36);
  \node[below,teal!80!black] at (.62,.36) {$\left(\frac8{11},\frac12\right)$};
  \draw[red!70!black,thick] (.75,.5) circle (2.2pt);
  \draw[red!70!black] (.75,.5)--(.96,.39);
  \node[below,red!70!black] at (.98,.39) {$\left(\frac34,\frac12\right)$};
\end{tikzpicture}
\caption{The two limiting points used in the convex hull argument.
The solid segment lies on $11\Re s+10\Re w=13$ and meets
$\Re w=1/2$ at $\Re s=8/11$, to the left of the secondary pole.
The connecting region is specified in the text.}
\label{fig:hull}
\end{figure}

The continuation descends to the multiplicative variables used in
\eqref{eq:coefficient}. On the original tube, the pole-free product
$C(s)G_4(s,w)N_4(s,w)\widehat A_4(s,w)$ is invariant under
\[
 (s,w)\mapsto(s+\pi i/\log q,w),\qquad
 (s,w)\mapsto(s,w+2\pi i/\log q).
\]
Each translate of its holomorphic continuation is defined on the same
connected tube over the convex hull and agrees with the original
function on a nonempty open subset. The identity theorem preserves
both periods. Since any two logarithmic preimages under
$(s,w)\mapsto(q^{-2s},q^{-w})$ differ by integer multiples of these periods, the continuation gives a single-valued
meromorphic function $\widehat A_4(u,v)$. At $u=0$ and
$v=q^{-1/2}$, it agrees with the original power series.

\section{The secondary residues}\label{sec:residues}
\subsection{The coefficients of the double poles}\label{sec:constants}
For each prime $P\in\F_q[T]$, set
\[\beta_P=\begin{cases}(1+p^{-2})^{-1},&\deg P\text{ odd},\\ p^2/(p+1)^2,&\deg P\text{ even},\end{cases}
\quad C_0=\prod_P(1-p^{-2}\beta_P)=P(q^{-2})>0.
\]
Define
\[r_P(x)=\frac{1-px+p^2(p-1)x^2}{1-x},\qquad
 E(t)=\prod_{\deg P\text{ even}}\left(1-\frac{2r_P(t^{\deg P})}{p(p+2)}\right).
\]
This product converges normally near $|t|=q^{-3/2}$. Let $t_0=\sqrt q$, $f=1-q^{-2}$, $E_+=E(q^{-3/2})$, $E_-=E(iq^{-3/2})$. Before multiplying by the Euler factors, the coefficients of the double poles are
\begin{align}
 A_1^{\rm model}&=-\frac{f(t_0^2+4t_0+1)}{16t_0},\\
 A_\zeta^{\rm model}&=\frac{f(t_0^2+1)}{16}
 \frac{1-\zeta/t_0}{1-\zeta t_0}\quad(\zeta\ne1,\ \zeta^4=1).
\end{align}
The coefficients for the family in \eqref{equation} are
\begin{equation}\label{eq:Aroot} A_\zeta=C_0E(\zeta q^{-3/2})A_\zeta^{\rm model}.\end{equation}
\begin{remark}\label{rem:nonzero}
In particular $A_1<0$.
For $r\in\{0,1,2,3\}$, put $L_r=\sum_{\zeta^4=1}A_\zeta\zeta^{-r}$ and $c=C_0f$. Then
\begin{align}
 L_0&=c(E_--E_+)/4,\\
 L_1&=-c\bigl((t_0+t_0^{-1}+2)E_+-(t_0-t_0^{-1})E_-\bigr)/8,\\
 L_2&=-c(E_++E_-)/4,\\
 L_3&=-c\bigl((t_0+t_0^{-1}+2)E_++(t_0-t_0^{-1})E_-\bigr)/8.
\end{align}
Since $\alpha_r=q^{3/2}L_r/3$, it suffices to show that all four $L_r$ are strictly negative: at $x=p^{-3/2}$,
\[r_P(-x)-r_P(x)=\frac{2x(p-1)^2}{p(1-x^2)}>0,
\]
so $0<E_-<E_+$. Indeed, the primes of even degree satisfy $p\ge q^2\ge9$, for which $r_P(\pm p^{-3/2})<3$; every Euler factor is therefore positive. There are irreducibles of degree $2$, ensuring that the inequality is strict. Without the arithmetic weight, the linear coefficient vanishes for $m\equiv0\pmod4$. The Euler factors in \eqref{eq:Aroot} make this coefficient nonzero.
\end{remark}
\subsection{The residue on \texorpdfstring{$N_4=0$}{N4=0}}\label{sec:derivative}
We now compute the restrictions needed for the constant coefficient of the secondary polynomial. Let $y=q^w$ and fix $\zeta^4=1$. On $Ry=\zeta$, for each prime $P$ write
\[v_P=p^w,\quad\eta=\zeta^{\deg P},\quad t_P=v_P/p,\quad
 c_P=\frac{1-\eta p^{-2}v_P}{1-\eta p^{-1}v_P}.
\]
For split $P$, let $E_P=1-p^{-1}v_P^{-2}$ and define, for $0\le a\le2$,
\[F_a=\frac{v_P^{-a}(1-p^{-1})^2+p^{-1}\eta^{a+1}v_P(1-v_P^{-2})^2}{E_P^2},
\qquad F_{3}=\frac{v_P^{-3}(1-p^{-1}v_P^2)^2}{E_P^2}.
\]
For inert $P$, let $E_P=1-p^{-2}v_P^{-4}$ and define
\[F_a=\frac{(1-p^{-2})v_P^{-a}+p^{-1}\eta^{a+1}(v_P-v_P^{-3})}{E_P},
\qquad F_{3}=\frac{v_P^{-3}(1-p^{-2}v_P^4)}{E_P}.
\]
In either case put
\begin{equation}\label{eq:JP}
 J_P(w,\zeta)=\frac{1-p^{-1}}{c_P}
 \left\{\frac{\sum_{a=0}^{3}t_P^aF_a}{1-t_P^4}
 +(1+p^{-1}\eta v_P)-F_0-\eta p^{-1}v_P^{-1}F_{2}\right\}.
\end{equation}
The terms $+(1+p^{-1}\eta v_P)-F_0$ give the contribution from numerator valuation zero. The last term comes from the sum over $D$ with coefficient $\mu(D)$ and occurs only when $P\nmid N$.

Let $H=G_4 N_4\widehat A_4$. Its restriction to the divisor $N_4=0$ is
\begin{equation}\label{eq:HN}
 H_{N,\zeta}(w)=\mathcal P(w)\frac{(1-y^{-4})J_4(y,\zeta/y)}{1-y}
 \frac{1-\zeta y/q}{1-\zeta y}\prod_PJ_P(w,\zeta).
\end{equation}
The product and its derivative converge normally near the intersection of $G_4=0$ and $N_4=0$. At $w=1/2$, $J_P=p^2/(p^2+1)$ when $P$ has odd degree, while for $P$ of even degree,
\[J_P=\frac{p(p+2)}{(p+1)^2}
 \left(1-\frac{2r_P(\eta p^{-3/2})}{p(p+2)}\right).
\]
Thus $\prod_PJ_P=C_0E(\zeta q^{-3/2})$, which gives~\eqref{eq:Aroot}.
\subsection{The residue on \texorpdfstring{$G_4=0$}{G4=0}}
Using $\Theta_{4,\zeta}$ from Section~\ref{sec:gauss-residue}, set
\[
 S_{4,\zeta}^{\gamma}(x)=\sum_N\gamma(N)\Theta_{4,\zeta}(N)x^{\deg N}.
\]
At $\zeta=1$ this is the notation $\Theta_4,S_4^\gamma$. For squarefree $H\in\F_q[T]$ whose prime divisors have even degree, let
\[\lambda(H)=\prod_{P\mid H}\frac{-2}{p+2},\qquad h=\deg H.
\]
The complete averages $c_H(j)$ and differences $d_{H,b}(j)$ are defined in \eqref{eq:complete-averages}. Their generating function is
\[\sum_jc_H(j)U^j=\frac{1-q\kappa_q U}{1-q^2\kappa_q U}
 \prod_{P\mid H}r_P((\kappa_q U)^{\deg P}).
\]
With $K_\zeta(a,j)$ as defined in Section~\ref{sec:gauss-residue}, the correction is
\begin{align}\label{eq:Rcorrection}
 \mathcal R_{H,\zeta}(x)
 &=\sum_{b\ge0}q^bx^{h+b}\sum_{b/2<j\le(b+h)/2}K_\zeta(h+b,j)d_{H,b}(j)\\
 &\quad-q^{-h}\sum_{a=0}^{h-1}(qx)^a\sum_{0\le j\le a/2}K_\zeta(a,j)c_H(j),\nonumber\\
 \mathcal R_\zeta(x)&=\sum_H\lambda(H)\mathcal R_{H,\zeta}(x).\nonumber
\end{align}
With $t_\zeta=\zeta q^{-1/2}x^2$, we obtain
\begin{equation}\label{eq:Sgamma}
 S_{4,\zeta}^{\gamma}(x)=
 \frac{J_4(qx,R_\zeta)}{(1-qx)(1-q^{-2}(qx)^4)}
 \frac{1-qt_\zeta}{1-q^2t_\zeta}E(t_\zeta)+\mathcal R_\zeta(x).
\end{equation}
The correction is holomorphic in $|x|<q^{-1/2+\eta_0}$ for some $\eta_0>0$, and its derivatives may be computed term by term. More precisely,
\[\sup_{|x|\le q^{-1/2+\eta_0}}|\mathcal R_{H,\zeta}(x)|
 \ll |H|^{-1/4+\eta_0+\eps_0/2+\eps},\quad
 \eta_0+\eps_0/2+2\eps<1/4.
\]
Section~\ref{app:residue} proves this estimate using the character sums subject to the trace congruence and the identity involving $L_{q^2}/L_q$ for the nonprincipal character in the sum over $B$. The factors arising from common divisors are included in the bounds for these character sums. Formula~\eqref{eq:Rcorrection} therefore determines the regular part used below.

We identify the residue on an open portion of $G_4=0$ before
continuing it to the central intersection. On $\kappa=5/4$, the
region obtained by coefficient interpolation contains $\tau<1/12$.
More precisely, after multiplication by $G_4$ and taking the finite
difference in the numerator degree, the sums with fixed degrees
$a=\deg N$ and $b=\deg D$ are bounded by
\[
 q^{(11/12+\eps)a+(13/6+\eps)b}.
\]
Multiplication by $x^a q^{-2\kappa b}$ gives
\[
 q^{(\tau-1/12+\eps)a+(-1/3+\eps)b}.
\]
The series of these sums converges normally in a neighborhood of
$(\sigma,\tau)=(27/20,-1/10)$ on the divisor. For fixed degrees there are
only finitely many $N,D$. The estimates are uniform in all imaginary
translates, and therefore apply to each of the four branches.
Normal convergence consequently allows us
to take residues term by term. The factor $1-qx$ used for the finite difference
is nonzero there. The residue sum over $D$ is absolutely convergent
and gives the Euler product in Section~\ref{app:residue}. Hence

\begin{equation}\label{eq:gausspolar}
 [G_4\widehat A_4]_{G_4=0}=
 \mathcal P(w)(1-q^{-2})C_0S_{4,\zeta}^{\gamma}(q^{w-1}).
\end{equation}
The prime-omission input and the branchwise local calculation used here are given next.

\subsection{Identification of the Gauss residue and convergence of its correction}\label{app:residue}
For a finite set $S$ of omitted extension primes, normalize the residue by
\[
 \Theta_{4,\zeta}^{(S)}(M)=\frac{1}{1-Q^{-1}}
 \lim_{z\to z_{4,\zeta}}(1-Q^5z^4)\Psi_{q^2}^{(S)}(M,z).
\]
Patterson's prime-omission identity \cite[p.~250]{Patterson} gives
\begin{equation}\label{eq:residue-omission}
 \Theta_{4,\zeta}^{(S\cup\{\pi\})}(M)
 =\frac{\Theta_{4,\zeta}^{(S)}(M)}{1+|\pi|_Q^{-1}},
 \qquad\pi\notin S.
\end{equation}
The identity is degreewise, with multiplier independent of the degree
class, and hence holds at every $z_{4,\zeta}$. It does not require
$M$ to be a unit at $\pi$; we use it at numerator valuations zero
and two. The following derivation retains every branch phase.

\begin{proposition}\label{prop:four-residue-phases}
For rational monic $N\ne0$ and rational squarefree monic $D$ with $(N,D)=1$,
\[
 \Theta_{4,\zeta}(ND^2)=\zeta^{3\deg D}|D|_q^{-1/2}\Theta_{4,\zeta}(N).
\]
\end{proposition}
\begin{proof}
First let $\pi$ be an extension prime, $d=\deg\pi$, $p=Q^d$, $(N,\pi)=1$,
and $\pi\notin S$. Set
\[
 g_3(N,\pi)=\sum_{x\bmod\pi}\chi_\pi(x)^3e_{q^2}(Nx/\pi)
 =\chi_\pi(N)g_3(1,\pi).
\]
At numerator valuation two, the only nonzero local modulus exponents are
zero and three, and
$G_{q^2}(N\pi^2,\pi^3)=p^2g_3(N,\pi)$.
Gauss CRT therefore gives the exact coefficientwise identity
\[
 \Psi_{q^2}^{(S)}(N\pi^2,z)
 =\Psi_{q^2}^{(S\cup\{\pi\})}(N\pi^2,z)
 +p^2g_3(N,\pi)z^{3d}\Psi_{q^2}^{(S\cup\{\pi\})}(N,z).
\]
The transformed numerator is initially $N\pi^8$; it can be replaced by
$N$ only inside the $\pi$-coprime series. No unrestricted off-pole
periodicity has been used. Taking residues and applying omission to
both terms gives
\begin{equation}\label{eq:valuation-two-residue}
 \Theta_{4,\zeta}^{(S)}(N\pi^2)
 =p^3z_{4,\zeta}^{3d}g_3(N,\pi)\Theta_{4,\zeta}^{(S)}(N).
\end{equation}
This rearrangement does not divide by a possibly vanishing residue.

Suppose $P=\pi\pi^{\sigma_0}$ is rational split of degree $2d$, so
$p=|P|_q=Q^d$. Apply \eqref{eq:valuation-two-residue} twice. The multiplier is
\[
 p^6z_{4,\zeta}^{6d}\chi_\pi(N)\chi_{\pi^{\sigma_0}}(N\pi^2)
 g_3(1,\pi)g_3(1,\pi^{\sigma_0}).
\]
For rational $N$, conjugation gives
$\chi_{\pi^{\sigma_0}}(N)=\chi_\pi(N)^{-1}$.
The resultant of the conjugate pair satisfies $R^q=(-1)^dR$, whence
$\chi_{\pi^{\sigma_0}}(\pi)=\kappa_q^d$ and
$\chi_{\pi^{\sigma_0}}(\pi^2)=1$.
Let $g_1$ denote the analogous Gauss sum with the first power of the
character. Trace compatibility gives
$g_3(1,\pi^{\sigma_0})=g_1(1,\pi)$.
Since $Q\equiv1\pmod8$, the inverse-character Gauss identity gives
$g_3(1,\pi)g_1(1,\pi)=p$.
Thus the multiplier is
\[
 p^7z_{4,\zeta}^{6d}=\zeta^{6d}p^{-1/2}
 =\zeta^{3\deg P}|P|_q^{-1/2}.
\]

Now let $P$ be inert of odd rational degree $d$. Put $r=q^d$, so its
extension norm is $r^2$, with residue field $K=\mathbb F_{r^2}$ and
rational subfield $k=\mathbb F_r$. The quartic character is trivial on
$k^\times$. With $\psi_k=\psi_q\circ\operatorname{Tr}_{k/\F_q}$, the additive
character is $\psi_k(\operatorname{Tr}_{K/k}(x/P'(\alpha)))$.
Since $P'(\alpha)\in k^\times$,
this scaling contributes no quartic phase.
Choose $\xi\ne0$ with trace zero. Then
\[
 \chi_P(\xi)=(-1)^{(r+1)/4}=\kappa_q^d=:\eta.
\]
The third-power character sums with a prescribed trace are $(r-1)\eta$ at zero
and $-\eta$ at every nonzero value. Consequently
\[
 g_3(1,P)=(r-1)\eta+(-\eta)(-1)=r\kappa_q^d.
\]
Also $\chi_P(N)=1$ for rational $N$ coprime to $P$. Equation
\eqref{eq:valuation-two-residue} now has multiplier
\[
 (r^2)^3z_{4,\zeta}^{3d}r\kappa_q^d
 =r^7\kappa_q^{4d}\zeta^{3d}Q^{-15d/4}
 =\zeta^{3d}r^{-1/2}.
\]
Iterate over the rational prime divisors of $D$. Each intermediate
numerator remains rational and coprime to the next prime, proving the claim.
\end{proof}

To recover the original restricted series, delete every extension
prime above the original $ND$, not merely above a reduced numerator.
Its rational-prime factors are precisely $\beta_P$ defined in
Section~\ref{sec:constants}. On the chosen branch
$q^{-2(s+w)}=\zeta q^{-5/2}$, so
\[
 |D|^{1-2(s+w)}=\zeta^{\deg D}|D|^{-3/2}.
\]
The phase cancels that in Proposition~\ref{prop:four-residue-phases},
since $\zeta^{4\deg D}=1$. Thus, for fixed $N$, the absolutely
convergent signed $D$ sum is
\[
 C_0\Theta_{4,\zeta}(N)
 \prod_{P\mid N}\frac{\beta_P}{1-p^{-2}\beta_P}
 =C_0\gamma(N)\Theta_{4,\zeta}(N).
\]
The local quotient is one at inert primes and $p/(p+2)$ at split
primes. The raw residue has the additional factor $1-q^{-2}$.
This fixed-$N$ computation is distinct from moving a polar limit
through the full $N,D$ sum. The normally convergent degree-grouped
sums preceding \eqref{eq:gausspolar} identify the latter residue
on an open portion of $G_4=0$; continuation along that divisor
then identifies it near the central intersection.

To prove~\eqref{eq:Rcorrection}, substitute the finite orthogonality identity of Section~\ref{app:fibres} into the finite expression for the residue. Extend the contribution from $V=0$ to all $a=2j+\delta$; summing over $\delta$ gives the first term of~\eqref{eq:Sgamma}. The second sum in~\eqref{eq:Rcorrection} subtracts the inserted degrees $a<h$. For the remaining contribution, estimating the sum over $B$ first gives
\[|d_{H,b}(j)|\ll q^{h-\delta+(1+\eps_0)j+\eps h},\qquad b=2j+\delta-h.
\]
The coefficients in the finite residue formula decay as $q^{-\delta/2}$, with constants depending on finitely many residue classes. For $|x|\le q^{-1/2+\eta_0}$, the geometric sums over $j$ and $\delta$ give
\[\mathcal R_{H,\zeta}(x)\ll q^{(-1/4+\eta_0+\eps_0/2+\eps)h}.
\]
The Euler coefficients of the complete average give the same bound for the finite subtraction. Since $\lambda(H)=\prod_{P\mid H}[-2/(p+2)]$, the weighted count of $H$ of degree $h$ is $q^{\eps h}$. The sum over $H$ therefore converges normally once $\eta_0+\eps_0/2+2\eps<1/4$. Cauchy's formula gives normal convergence of the derivatives on smaller compact disks, justifying the derivative in~\eqref{eq:Bformula}.

Let $w_0=1/2$, $u_\zeta=\zeta q^{-3/2}$, and choose $s_\zeta$ with $q^{-2s_\zeta}=u_\zeta$. Set
\begin{align}
 H_{G,\zeta}(w)&=(1-q^{4(w-w_0)})\,
 \mathcal P(w)(1-q^{-2})C_0S_{4,\zeta}^{\gamma}(q^{w-1}),\\
 A_\zeta&=H(s_\zeta,w_0)/16,\\
 B_\zeta&=3A_\zeta+
 \frac{H_{N,\zeta}'(w_0)-H_{G,\zeta}'(w_0)}{16\log q}.\label{eq:Bformula}
\end{align}
The apparent singularity of $H_G$ is removable. Along the two divisors,
$s_G'(w)=-1$ and $s_N'(w)=-1/2$, so $H_s=2(H_N'-H_G')$. Expanding $(1-(u/u_\zeta)^4)^{-2}$ gives~\eqref{eq:Bformula}. Thus the secondary polynomial in~\eqref{eq:main-result} can also be written as
\[\sum_{\zeta^4=1}\zeta^{-m}\left((m+4)A_\zeta+
 \frac{H_{N,\zeta}'(1/2)-H_{G,\zeta}'(1/2)}{16\log q}\right).
\]
Multiplying this polynomial by $q^{3m/2}$ and putting $m=g/3+1$ gives the secondary term in \eqref{eq:main-result}.
\begin{remark}
The identity on $N_4=0$ is obtained from the already continued
function $H=G_4N_4\widehat A_4$. On an open portion of this divisor inside
$R_1$, the calculation of the complete average gives
$H|_{N_4=0}=H_{N,\zeta}(w)$: multiplication by $N_4$ makes the
correction vanish there. On the positive real branch this portion
contains $5/8<\sigma<3/4$ and $\tau=2-2\sigma$.
The Euler product defining $H_{N,\zeta}$ continues along the
connected divisor to the central intersection. Apply the identity theorem with the factor $C(s)$ retained, and
then divide where $C(s)\ne0$. This identifies the restriction of
$H$ with the Euler product near the central intersection and gives
\eqref{eq:Aroot}. The same argument
applies to each of the four branches. For $G_4=0$, the normally convergent identity continues
meromorphically along that divisor, again retaining $C(s)$ in the
identity theorem. The intervening counting intersection is
$(s,w)=(1,1/4)$ on the positive real branch; it is not claimed to
be a holomorphic point of $H$. Since $C(s)\ne0$ near the central
intersection, this gives \eqref{eq:gausspolar} there. \end{remark}

\section{Proof of Theorem~\ref{main} and Corollary~\ref{non}}\label{sec:proof}
\subsection{The first moment}
\begin{proof}[Proof of Theorem~\ref{main}]
Put $m=g/3+1$. At $w=1/2$, both $G_4$ and $N_4$ become
$1-q^6u^4$. By Section~\ref{sec:continuation}, the only possible
poles in the continued central slice are the simple pole $u=q^{-2}$
and double poles $u_\zeta=\zeta q^{-3/2}$, where $\zeta^4=1$.
Section~\ref{sec:residues} gives the principal parts
\[
 \widehat A_4(u,q^{-1/2})
 =\frac{A_\zeta}{(1-u/u_\zeta)^2}
 +\frac{B_\zeta}{1-u/u_\zeta}+O(1)
 \quad(u\longrightarrow u_\zeta).
\]
At $u=q^{-2}$, the principal part is
$q^{-2}C_q/(1-q^2u)$ by Section~\ref{PZ}.
In \eqref{eq:coefficient}, move the circle to
\[
 |u|=q^{-16/11-\delta},\qquad 0<\delta<1/22.
\]
The new circle lies inside the continued domain and contains no pole.
Its contribution is $O_{q,\delta}(q^{(16/11+\delta)m})$.
The crossed principal parts contribute
\[
 C_q q^{2m-2}
 +q^{3m/2}\sum_{\zeta^4=1}\zeta^{-m}
       \bigl((m+1)A_\zeta+B_\zeta\bigr).
\]
Since $m=g/3+1$, the first term is $C_q q^{2g/3}$ and the
second is $(\alpha_rg+\beta_r)q^{g/2}$, with
$r\equiv m\pmod4$. Choose $\delta<3\eps$ and absorb the
fixed factor depending on $q,\delta$ into the implicit constant.
This gives the error in \eqref{eq:main-result}. The sign
$\alpha_r<0$ follows from Remark~\ref{rem:nonzero}.
\end{proof}

\subsection{Non-vanishing}
\begin{proof}[Proof of Corollary~\ref{non}]
Let $\mathcal N_4(g)$ be the number of nonzero central values in
the family. Since $C_q>0$ and both remaining terms are of smaller
order, Theorem~\ref{main} gives
\[
 q^{2g/3}\ll_q
 \left|\sum_{\substack{F\in\Htwo\\(F,F^{\sigma_0})=1}}
 L_q(1/2,\chi_F)\right|
 \le \sum_{\substack{F\in\Htwo\\(F,F^{\sigma_0})=1}}
 |L_q(1/2,\chi_F)|.
\]
By Lemma~\ref{LLH}, for every $\eta>0$ the last sum is at most
$O_{q,\eta}(\mathcal N_4(g)q^{\eta(2g/3+2)})$.
Taking $\eta$ sufficiently small in terms of $\eps$ proves
$\mathcal N_4(g)\gg_{q,\eps}q^{(2/3-\eps)g}$.
\end{proof}

\section*{Acknowledgement}
The author thanks Professor Peng Gao for suggesting this problem
and for his guidance and helpful comments.
\section*{Statements and Declarations}

\noindent\textbf{Competing Interests.}
The author declares that there are no competing interests.

\begingroup\raggedright\microtypesetup{expansion=false}

\endgroup
\end{document}